\documentclass[reqno,12pt]{amsart}
\usepackage[margin=3cm]{geometry}
\input{packages}
\bibliography{bibl_min}

\numberwithin{equation}{section}

\DeclareMathAlphabet{\mathpzc}{OT1}{pzc}{m}{it}

\makeatletter

\begin{document}

\def\subsectionautorefname{Section}
\def\subsubsectionautorefname{Section}
\def\sectionautorefname{Section}
\def\equationautorefname~#1\null{(#1)\null}
\setcounter{tocdepth}{1}

\newcommand{\mynewtheorem}[4]{
  \if\relax\detokenize{#3}\relax 
    \if\relax\detokenize{#4}\relax 
      \newtheorem{#1}{#2}
    \else
      \newtheorem{#1}{#2}[#4]
    \fi
  \else
    \newaliascnt{#1}{#3}
    \newtheorem{#1}[#1]{#2}
    \aliascntresetthe{#1}
  \fi
  \expandafter\def\csname #1autorefname\endcsname{#2}
}

\newtheorem{theorem}{Theorem}[section]
\newtheorem{lemma}[theorem]{Lemma}
\newtheorem{lem}[theorem]{Lemma}
\newtheorem{rem}[theorem]{Remark}
\newtheorem{prop}[theorem]{Proposition}
\newtheorem{cor}[theorem]{Corollary}
\newtheorem{definition}[theorem]{Definition}
\newtheorem{example}[theorem]{Example}
\newtheorem{rmk}[theorem]{Remark}
\newtheorem{rmk*}[theorem]{Notation}


\def\defbb#1{\expandafter\def\csname b#1\endcsname{\mathbb{#1}}}
\def\defcal#1{\expandafter\def\csname c#1\endcsname{\mathcal{#1}}}
\def\deffrak#1{\expandafter\def\csname frak#1\endcsname{\mathfrak{#1}}}
\def\defop#1{\expandafter\def\csname#1\endcsname{\operatorname{#1}}}
\def\defbf#1{\expandafter\def\csname b#1\endcsname{\mathbf{#1}}}

\makeatletter
\def\defcals#1{\@defcals#1\@nil}
\def\@defcals#1{\ifx#1\@nil\else\defcal{#1}\expandafter\@defcals\fi}
\def\deffraks#1{\@deffraks#1\@nil}
\def\@deffraks#1{\ifx#1\@nil\else\deffrak{#1}\expandafter\@deffraks\fi}
\def\defbbs#1{\@defbbs#1\@nil}
\def\@defbbs#1{\ifx#1\@nil\else\defbb{#1}\expandafter\@defbbs\fi}
\def\defbfs#1{\@defbfs#1\@nil}
\def\@defbfs#1{\ifx#1\@nil\else\defbf{#1}\expandafter\@defbfs\fi}
\def\defops#1{\@defops#1,\@nil}
\def\@defops#1,#2\@nil{\if\relax#1\relax\else\defop{#1}\fi\if\relax#2\relax\else\expandafter\@defops#2\@nil\fi}
\makeatother

\defbbs{ZHQCNPALRVWS}
\defcals{ABCDOPQMNXYLTRAEHZKCIV}
\deffraks{apijklgmnopqueRCc}
\defops{IVC, PGL,SL,mod,Spec,Re,Gal,Tr,End,GL,Hom,PSL,H,div,Aut,rk,Mod,R,T,Tr,Mat,Vol,MV,Res,Hur, vol,Z,diag,Hyp,hyp,hl,ord,Im,ev,U,dev,c,CH,fin,pr,Pic,lcm,ch,td,LG,id,Sym,Aut,Log,tw,irr,discrep,BN,NF,NC,age,hor,lev,ram,NH,av,app,Quad,Stab,Per,Nil,Ker,EG,CEG,PB,Conf,MCG,Diff}
\defbfs{uvzwp} %

\def\ep{\varepsilon}
\def\ve{\varepsilon}
\def\abs#1{\lvert#1\rvert}
\def\dd{\mathrm{d}}
\def\WP{\mathrm{WP}}
\def\inj{\hookrightarrow}
\def\eq{=}

\def\i{\mathrm{i}}
\def\e{\mathrm{e}}
\def\st{\mathrm{st}}
\def\ct{\mathrm{ct}}
\def\rel{\mathrm{rel}}
\def\odd{\mathrm{odd}}
\def\even{\mathrm{even}}

\def\uC{\underline{\bC}}
\def\ol{\overline}

\def\Vrel{\bV^{\mathrm{rel}}}
\def\Wrel{\bW^{\mathrm{rel}}}
\def\twolev{\mathrm{LG_1(B)}}

\def\be{\begin{equation}}   \def\ee{\end{equation}}     \def\bes{\begin{equation*}}    \def\ees{\end{equation*}}
\def\ba{\be\begin{aligned}} \def\ea{\end{aligned}\ee}   \def\bas{\bes\begin{aligned}}  \def\eas{\end{aligned}\ees}
\def\={\;=\;}  \def\+{\,+\,} \def\m{\,-\,}

\newcommand*{\Tw}[1][\Lambda]{\mathrm{Tw}_{#1}}  
\newcommand*{\sTw}[1][\Lambda]{\mathrm{Tw}_{#1}^s}  

\newcommand{\HH}{{\mathbb H}}
\newcommand{\MM}{{\mathbb M}}
\newcommand{\bbC}{{\mathbb C}}
\newcommand{\TT}{{\mathbb T}}

\newcommand\PP{\mathbb P}
\renewcommand\R{\mathbb R}
\renewcommand\Z{\mathbb Z}
\newcommand\N{\mathbb N}
\newcommand\Q{\mathbb Q}
\renewcommand{\H}{\mathbb{H}}
\newcommand{\halfplane}{\mathbb{H}}
\newcommand{\chalfplane}{\overline{\mathbb{H}}}
\newcommand{\bk}{k}

\newcommand{\bfa}{{\mathbf a}}
\newcommand{\bfb}{{\mathbf b}}
\newcommand{\bfd}{{\mathbf d}}
\newcommand{\bfe}{{\mathbf e}}
\newcommand{\bff}{{\mathbf f}}
\newcommand{\bfg}{{\mathbf g}}
\newcommand{\bfh}{{\mathbf h}}
\newcommand{\bfm}{{\mathbf m}}
\newcommand{\bfn}{{\mathbf n}}
\newcommand{\bfp}{{\mathbf p}}
\newcommand{\bfq}{{\mathbf q}}
\newcommand{\bft}{{\mathbf t}}
\newcommand{\bfP}{{\mathbf P}}
\newcommand{\bfR}{{\mathbf R}}
\newcommand{\bfU}{{\mathbf U}}
\newcommand{\bfu}{{\mathbf u}}
\newcommand{\bfx}{{\mathbf x}}
\newcommand{\bfz}{{\mathbf z}}

\newcommand{\bfl}{{\boldsymbol{\ell}}}
\newcommand{\bfmu}{{\boldsymbol{\mu}}}
\newcommand{\bfeta}{{\boldsymbol{\eta}}}
\newcommand{\bftau}{{\boldsymbol{\tau}}}
\newcommand{\bfomega}{{\boldsymbol{\omega}}}
\newcommand{\bfsigma}{{\boldsymbol{\sigma}}}
\newcommand{\bfnu}{{\boldsymbol{\nu}}}
\newcommand{\bfrho}{{\boldsymbol{\rho}}}
\newcommand{\bfone}{{\boldsymbol{1}}}

\newcommand\cl{\mathcal}
\newcommand{\calH}{\mathcal{H}}

\newcommand{\calA}{\mathcal A}
\newcommand{\calK}{\mathcal K}
\newcommand{\calD}{\mathcal D}
\newcommand{\calC}{\mathcal C}
\newcommand\C{\mathcal C}
\newcommand{\calT}{\mathcal T}
\newcommand{\calB}{\mathcal B}
\newcommand{\calF}{\mathcal F}
\newcommand{\calV}{\mathcal V}
\newcommand{\calQ}{\mathcal Q}
\newcommand{\calX}{\mathcal X}
\newcommand{\calY}{\mathcal Y}
\newcommand{\calP}{\mathcal P}
\newcommand{\calJ}{\mathcal J}
\newcommand{\calS}{\mathcal S}
\newcommand{\calZ}{\mathcal Z}

\newcommand\pvd{\operatorname{pvd}}
\newcommand\per{\operatorname{per}}
\newcommand\thick{\operatorname{thick}}
\newcommand\rep{\operatorname{rep}}

\newcommand\perf{\operatorname{perf}}
\newcommand{\modules}{\operatorname{mod}}
\newcommand{\Modules}{\operatorname{Mod}}
\newcommand{\Loc}{\operatorname{Loc}}
\renewcommand{\Mod}{\operatorname{Mod}}
\renewcommand{\Hom}{\operatorname{Hom}}
\newcommand{\Ext}{\operatorname{Ext}}
\newcommand{\coker}{\operatorname{coker}}

\newcommand\Rep{\operatorname{Rep}}
\newcommand\ext{\operatorname{ext}}
\newcommand{\heart}{\heartsuit}

\newcommand{\sph}{\operatorname{sph}}
\newcommand{\Br}{\operatorname{Br}}
\renewcommand{\Tw}{\operatorname{Tw}}
\newcommand{\RHom}{\operatorname{RHom}}

\newcommand{\torsion}{\mathcal{T}}
\newcommand{\torsionfree}{\mathcal{F}}
\newcommand{\tstr}{\mathcal{L}}

\newcommand\cy{\mathrm{CY}}
\newcommand\CY{\mathrm{CY}}

\newcommand{\sgn}{\operatorname{sgn}}
\renewcommand{\GL}{\operatorname{GL}}
\renewcommand{\rk}{\operatorname{rank}}
\newcommand{\opL}{\operatorname{L}}

\newcommand\stab{\operatorname{Stab}}
\newcommand\gstab{\operatorname{GStab}}

\newcommand{\spn}{\operatorname{span}}
\newcommand{\GStab}{\operatorname{GStab}}
\newcommand{\PGStab}{\bP\mathrm{GStab}}
\newcommand{\MStab}{\operatorname{MStab}}
\newcommand{\PMStab}{\bP\mathrm{MStab}}
\newcommand{\Tilt}{\operatorname{Tilt}}

\def\oQ{\overline{Q}}
\def\bY{\mathbf{Y}}
\def\bX{\mathbf{X}}
\newcommand{\fmu}{\mu^\sharp}
\newcommand{\bmu}{\mu^\flat}
\newcommand{\Cone}{\operatorname{Cone}}
\newcommand\add{\operatorname{Add}}
\newcommand\Irr{\operatorname{Irr}}

\newcommand{\EGs}{\EG^{s}} 
\newcommand{\SEG}{\operatorname{SEG}} 
\newcommand{\pSEG}{\operatorname{pSEG}} 
\newcommand{\EGp}{\EG^\circ}       
\newcommand{\SEGp}{\SEG^\circ}       
\newcommand{\EGb}{\EG^\bullet}       
\newcommand{\SEGb}{\SEG^\bullet}       
\newcommand{\SEGV}{\SEG_{{^\perp}\calV}} 
\newcommand{\pSEGV}{\pSEG_{{^\perp}\calV}}
\newcommand{\SEGVb}{\SEGb_{{^\perp}\calV}} 
\newcommand{\pSEGVb}{\pSEG^\bullet_{{^\perp}\calV}}

\newcommand{\EGT}{\EG^\circ}
\newcommand{\uEG}{\underline{\EG}} 
\newcommand{\uCEG}{\underline{\CEG}} 

\newcommand{\CA}{\operatorname{CA}}
\newcommand{\OA}{\operatorname{OA}}
\newcommand{\wOA}{\widetilde{\OA}}
\newcommand{\wCA}{\widetilde{\CA}}
\newcommand{\BT}{\operatorname{BT}}
\newcommand{\SBr}{\operatorname{SBr}}
\newcommand\Bt[1]{\operatorname{B}_{#1}}
\newcommand\bt[1]{\operatorname{B}_{#1}^{-1}}

\newcommand{\Sim}{\operatorname{Sim}}
\def\dual{\iota}
\def\ivc{\iota_v}
\newcommand\ind{\operatorname{index}}
\newcommand{\Qgrad}{\overline{Q}}
\newcommand\diff{\operatorname{d}}
\def\Dsow{\Dcol}
\def\pvc{e \Gamma e}
\def\Qvc{Q_V^c}
\def\Gwt{\FGamma_\wt}
\def\wX{\widetilde{X}}

\def\DD{\mathbf{D}}
\def\uD{\underline{\calD}}
\def\uh{\underline{\calH}}
\def\uZ{\underline{Z}}
\def\us{\underline{\sigma}}
\newcommand{\isom}{\cong}

\newcommand{\tilt}[3]{{#1}^{#2}_{#3}}
\newcommand\Stap{\Stab^\circ} 
\newcommand\Stas{\Stab^\bullet}
\newcommand{\cub}{\operatorname{U}} 
\newcommand{\skel}{\wp} 

\newcommand{\mai}{\mathbf{i}} 

\newcommand\wT{\widetilde{\TT}} 
\newcommand{\Tri}{\Delta}
\newcommand{\deco}{\Delta}

\def\w{\mathbf{w}}
\def\wtpt{\w_{+2}}
\newcommand\surf{\mathbf{S}}  
\newcommand\surfo{{\mathbf{S}}_\Tri}  
\newcommand\surfw{\surf^\w}  
\newcommand\sow{\surf_\w}  
\newcommand\subsur{\Sigma}  
\newcommand\colsur{\overline{\surf}_\w}  
\def\Dsan{\calD_3(\surfo)}
\def\Dsow{\calD(\sow)}
\def\Dsub{\calD_3(\subsur)}
\def\Dcol{\calD(\colsur)}

\def\ww{node[white]{$\bullet$}node[red]{$\circ$}}
\def\nn{node{$\bullet$}}

\newcommand{\uk}{\mathbf{k}}
\def\sing{\operatorname{Sing}}

\newcommand\jiantou{edge[->]}
\newcommand\AS{\mathbb{A}}

\def\grad{\lambda}
\def\gms{\surf^\grad}
\def\gmsw{\surf^{\grad,\wt}_{\Tri}}
\def\iT{\TT_0}

\def\weta{\widetilde{\eta}}
\def\wzeta{\widetilde{\zeta}}
\def\walpha{\widetilde{\alpha}}
\def\wbeta{\widetilde{\beta}}
\def\wgamma{\widetilde{\gamma}}

\def\dsan{\calD_3(\surfo)}
\def\psan{\per(\surfo)}
\def\dsow{\Dsow}
\def\psow{\per(\colsur)}

\def\RP{\operatorname{RP}}
\def\Rs{\operatorname{RS}}
\newcommand{\ST}{\operatorname{ST}}  
\newcommand{\STp}{\ST^\circ}  

\newcommand{\Int}{\operatorname{Int}}

\def\hori{_{\operatorname{H}}}

\newcommand{\Note}[1]{\textcolor{red}{#1}}
\newcommand{\note}[1]{\textcolor{Emerald}{#1}}
\newcommand{\qy}[1]{\textcolor{cyan}{#1}}

\newcommand\foli[5]{
	\foreach \m in {0,#1,...,#5}{
		\coordinate (#2#4) at ($(#2)!.5!(#4)$);
		\draw[Emerald!50]plot [smooth,tension=.5] coordinates
		{(#2) ($(#2#4)!\m!(#3)$) (#4)};
		\draw[thick,gray](#2)to(#3)to(#4);
}}

\def\iA{\AS_0}
\def\iA{\AS_0}

\newcommand\Ind{\operatorname{Ind}}

\newcommand\xx{\mathbf{X}} 
\newcommand\surp{\xx^\circ}
\newcommand{\Zer}{\operatorname{Zero}}
\newcommand{\Pol}{\operatorname{Pol}}
\newcommand{\Crit}{\operatorname{Crit}}
\newcommand{\FQuad}{\operatorname{FQuad}}
\newcommand{\FQuab}{\FQuad^{\bullet}}

\newcommand{\Imgy}{\operatorname{Im}}
\newcommand{\numarc}{n}


\newcommand{\Exch}{\operatorname{Exch}}
\newcommand{\arrowIn}{
	\tikz \draw[-stealth] (-1pt,0) -- (1pt,0);
}

\newcommand{\cal}[1]{\mathcal{#1}}

\newlength{\halfbls}\setlength{\halfbls}{.8\baselineskip}

\newcommand*{\Teichmuller}{Teich\-m\"uller\xspace}

\newcommand\bra{\langle}
\newcommand\ket{\rangle}
\newcommand{\del}{\partial}
\newcommand{\deld}[1]{\frac{\del}{\del {#1} }}
\newcommand{\frap}{\frac{1}{2\pi\ii}}
\renewcommand{\Re}{\operatorname{Re}}
\renewcommand{\Im}{\operatorname{Im}}
\newcommand{\Rp}{\mathbb{R}_{>0}}
\newcommand{\Rm}{\R_{<0}}

\def\lift{\mathrm{lift}}
\def\ul{\underline}
\newcommand\<{\langle}
\renewcommand\>{\rangle}
\def\h{\calH}
\def\D{\calD}
\def\aut{\mathpzc{Aut}}
\tikzcdset{arrow style=tikz, diagrams={>=stealth}}
\usetikzlibrary{arrows}

\def\acf{\mathbf{k}}
\def\PTS{\mathbb{P}T\surf}
\newcommand\coho[1]{\operatorname{H}^{#1}}
\def\MTS{\mathbb{R}T\surf^{\lambda}}

\def\ora{\overrightarrow}
\def\harc{\ora{\gamma_h}}
\def\varc{\ora{\gamma_v}}
\def\cube{\mathrm{U}}

\newcommand{\on}[1]{\operatorname{#1}}
\newcommand{\iv}[1]{(#1)^{-1}}


\title[Verdier quotients of 3-Calabi-Yau quiver categories]{Verdier quotients of Calabi-Yau categories from quivers with potential}

\author{Anna Barbieri}
\address{A.B.:  Dipartimento di Informatica, 
Universit\`a di Verona, 
Strada Le Grazie 15, 37134 Verona - Italy }
\email{anna.barbieri@univr.it}
\thanks{Research of A.B.\ was supported by the
project \emph{SQUARE - Structures for Quivers, Algebras and Representations}, financed 
PRIN-2022, MUR}
\author{Yu Qiu}
\address{Qy:
Yau Mathematical Sciences Center and Department of Mathematical Sciences,
Tsinghua University,
100084 Beijing, China.
\&
Beijing Institute of Mathematical Sciences and Applications, Yanqi Lake, Beijing, China}
\email{yu.qiu@bath.edu}
\thanks{Research of Qy was supported by National Natural 
Science Foundation of China (Grant No. 12425104 and 12031007) and 
National Key R\&D Program of China (No. 2020YFA0713000).
}

\begin{abstract}
We investigate a class of triangulated categories obtained as Verdier quotients of 3-Calabi-Yau categories combinatorially described by quivers with potential from (decorated) marked surfaces. We study their bounded t-structures and consider in particular the exchange graphs of hearts and silting objects respectively, and show that the Koszul isomorphism between these graphs is preserved under Verdier quotient.
\end{abstract}
\maketitle
\tableofcontents

\section{Introduction}
\label{sec_intro}

In this paper we study a class of triangulated categories which 
arise as Verdier quotient of Calabi-Yau categories from a differential graded algebra of 
a quiver with potential $(Q,W)$, 
and which can be associated with some mixed-angulation of a marked 
Riemann surface with boundary. 
\par\medskip

For any differential graded (dg) algebra $\Lambda$, we denote 
by $\calD(\Lambda)$, $\per(\Lambda)$, and $\pvd(\Lambda)$ the derived category, the 
perfect derived category, and the perfectly valued derived 
category of $\Lambda$ respectively. 
Let $e$ be an idempotent of a non-positive dg algebra $\Gamma$. Triangulated categories 
over $e\Gamma e$ appear for instance in the context of 
singularity categories in \cite{kalckyang1,kalckyang2}, and 
in the context of weighted marked surfaces in \cite{BMQS}. 
Here we let $\Gamma=\Gamma(Q,W)$ be the 
Ginzburg dg algebra of Calabi-Yau dimension $3$ from a quiver 
with potential $(Q,W)$, and fix $e=\sum_{j\not\in I}e_j$, the sum of trivial 
paths $e_i$ associated with the complement of a subset $I$ of 
the vertices $Q_0$ of $Q$. Let $\Gamma_I$ be the Ginzburg dg algebra associated with the restriction of $(Q,W)$ to $I$. 
It is known that $\calD(e\Gamma e)$ and $\pvd(e\Gamma e)$ are 
Verdier quotients of $\calD(\Gamma)$ and $\pvd(\Gamma)$ by $\calD(\Gamma_I)=\calD(\Gamma/\Gamma e\Gamma)$ and $\pvd(\Gamma_I)=\pvd(\Gamma/\Gamma e\Gamma)$ 
respectively \cite{kalckyang2},
\[0 \to \calD(\Gamma/\Gamma e\Gamma)\to \calD(\Gamma) \to \calD(e\Gamma e) \to 0,\] 
\[0 \to \pvd(\Gamma/\Gamma e\Gamma)\to \pvd(\Gamma) \to \pvd(e\Gamma e)  \to 0.\] 
While the first class fits into a recollement of triangulated 
categories 
induced by idempotents which is pretty well studied (see for instance 
\cite{psarou} and \cite{kalckyang2}), the latter are 
less known and a priori much less 
well-behaved triangulated categories. 
\par\medskip
The goal of this paper is to provide a comprehensive description of 
the derived categories of $e\Gamma e$ and to study 
properties of $\pvd(\Gamma)$ that descend to 
$\pvd(e\Gamma e)$. 
We recall some preliminaries in Section \ref{sec_prelim}. In 
Section \ref{sec_loc_pvd} we 
introduce $\calD(e\Gamma e)$, $\per(e\Gamma e)$, $\pvd(e\Gamma e)$, 
describe their relations, and we adapt the gluing method of bounded 
t-structures in a recollement, by Beilinson-Bernstein-Deligne, to the sequence of perfectly 
valued categories. 
In Sections \ref{sec_DMS}, 
\ref{sec_tstructures}, and \ref{sec_duality}, we mainly restrict our 
attention to quivers with potential coming from triangulations of marked Riemann 
surfaces with boundary components. 
We are interested in bounded t-structures of $\pvd(e\Gamma e)$ 
regarded as a 
Verdier quotient and in Happel-Reiten-Smal\o\,(HRS) tilts 
between them. The notions of bounded t-structures 
and a special case of tilting (simple tilts) are compactly 
encoded in a graph, and we study a HRS-tilting exchange 
graph of $\pvd(e\Gamma e)$. After recalling simple-projective 
duality for $\Gamma$, we state a correspondence between tilting 
and silting exchange graphs associated with $e\Gamma e$, 
obtaining a \emph{relative} simple-projective correspondence,  
that is a duality between length hearts 
with finitely many simples and silting 
objects at the level of 
quotient triangulated categories. 
This kind of correspondence is also known as 
Koenig-Yang correspondence.
\par \medskip
In the rest of the Introduction we summarize the main results. 
We refer the reader to the list of notation at 
page \pageref{list_of_notation} for the relevant definitions. 

\subsection*{Bounded t-structures under Verdier quotient}

The exchange graph $\EG(\calD)$ of a triangulated category $\calD$ is the graph 
whose vertices are hearts of bounded t-structures on $\calD$ and 
whose directed edges  
correspond to simple forward HRS-tilts, i.e. 
tilts at a torsion-free class generate by a 
simple object. A complete description of 
this graph might be difficult to obtain, so it is 
standard to rectrict our 
attention to a connected component containing a fixed 
distinguished heart as a vertex. In this paper we use the 
terminology finite heart to mean the heart of a bounded 
t-structure which is finite lenght and has finitely many 
simple objects, up to isomorphism.
\par
Let $\Gamma=\Gamma(Q,W)$ be the Ginzburg algebra of a 
quiver with potential $(Q,W)$ dual to 
a triangulation of a decorated marked surface. 
The exchange graph $\EGp\left(\pvd(\Gamma)\right)$ is the connected 
component of the exchange graph of $\pvd(\Gamma)$ containing 
the bounded t-structure with heart $\rep(Q,W)$ as a vertex. 
It was studied in \cite{kq} and is related 
with cluster theory and the exchange graph of 
triangulations of a (decorated) marked 
surface (\cite{qiubraid16}). 
\par 
We consider the category $\pvd(e\Gamma e)$ regarded, as we 
have seen above, as the 
quotient $\pvd(\Gamma)/\pvd(\Gamma_I)$, and we define 
the subgraph 
$\EG^\bullet(\pvd(e\Gamma e))$ of $\EG(\pvd(e\Gamma e))$, 
whose vertices arise 
from localization of finite hearts of $\pvd(\Gamma)$ 
in $\EGp\left(\pvd\Gamma\right)$. We show that it can be 
read off from $\EG(\pvd\Gamma)$, and that any simple forward tilt 
of a finite heart in $\EG^\bullet(\pvd(e\Gamma e))$ \lq\lq lifts\rq\rq\ 
to a simple forward tilt in $\pvd(\Gamma)$. This relies on the 
construction of $(Q,W)$ as the dual quiver of the triangulation of 
a marked surface. 
Denote by 
$\EGp(\pvd\Gamma, \pvd\Gamma_I)$ the full subgraph 
of $\EGp(\pvd\Gamma)$ whose hearts induce a bounded 
t-structure on $\pvd(e\Gamma e)$. 
\begin{theorem}[{Theorem \ref{quot_graph}}]\label{thm_intro1}
    The graph $\EG^\bullet(\pvd(e\Gamma e))$ 
    is isomorphic to the quotient of the relative exchange 
    graph $\EGp(\pvd\Gamma, \pvd\Gamma_I)$ obtained by contracting edges 
    labelled by a simple object in $\pvd(\Gamma_I)$
    \[ 
    \EG^\circ(\pvd\Gamma)\supset \EG^\circ(\pvd\Gamma,\pvd\Gamma_I) 
    \twoheadrightarrow \EG^\bullet(\pvd(e\Gamma e)).
    \]
It consists of connected components of $\EG(\pvd(e\Gamma e))$.
\end{theorem}
At this stage, we are 
not able to exclude the existence of spurious 
components of $\EG(\pvd(e\Gamma e))$. 
The vertices of $\EGb(\pvd(e\Gamma e))$ are finite hearts of bounded 
t-structures in $\pvd(e\Gamma e)$ of the form $\cl H/\cl H_I$ for 
some finite heart $\cl H$ of $\pvd(\Gamma)$ and Serre 
subcategory $\cl H_I=\cl H\cap\pvd\Gamma_I$. They are 
module categories over finite dimensional algebras whose 
underlying quiver is deduced explicitly from $(Q,W)$ and $I$. 
\par\smallskip
A natural question is how t-structures obtained by a generic tilt 
at a torsion pair in a heart of $\pvd(e\Gamma e)$ look like. 
Reachable bounded t-structures of $\pvd(e \Gamma e)$, i.e., those 
that are intermediate with respect to a t-structure in 
$\EG^\bullet(\pvd(e\Gamma e))$, turn out to be equivalent to 
abelian quotients of some bounded heart in $\pvd(\Gamma)$. 

\begin{prop}[{Corollary \ref{any_heart_quot}}]\label{thms_intro2} 
Let $\ol{\cl H} = \cl H/(\cl H\cap \pvd\Gamma_I)\in \EGb(\pvd (e\Gamma e))$ and 
$\mu_f^\sharp\ol{\cl H}$ be the heart of another bounded 
t-structure in $\pvd(e\Gamma e)$ obtained by tilting $\ol{\cl H}$ at a 
torsion free class $f$. 
Then there exists a torsion-free class $\torsionfree\subset \cl H$ 
and $\cl H'= \mu^\sharp_\torsionfree\cl H$ in $\pvd\Gamma$ 
such that 
$\mu_f^\sharp\ol{\cl H} = \cl H' /\left(\cl H' \cap\pvd\Gamma_I\right)$.
    \end{prop}
This depends on the existence of recollements of abelian and 
triangulated categories at the level of module categories over $J:=H^0\Gamma$ 
and derived categories over $\Gamma$ 
\[
\xymatrix@C=2cm{
\calD(\Gamma/\Gamma e\Gamma) \ar[r]^{\iota} 
& \calD(\Gamma) \ar[r]^{\pi}\ar@/^1pc/[l]_{p}\ar@/_1pc/[l]_{q} 
&  \calD(e\Gamma e)\ar@/^1pc/[l]_{r}\ar@/_1pc/[l]_{\ell} \\
\pvd(\Gamma/ \Gamma e\Gamma) \ar@{}[u]|{\bigcup} \ar[r]^{\iota} & 
\pvd(\Gamma) \ar[r]^{\pi} \ar@{}[u]|{\bigcup} & \pvd(e\Gamma e) \ar@{}[u]|{\bigcup}\\
\modules (J/JeJ) \ar[r]^{\iota} \ar@{}[u]|{\bigcup}
& \modules(J) \ar[r]^{\pi}\ar@/^1pc/[l]_{p}\ar@/_1pc/[l]_{q} \ar@{}[u]|{\bigcup}
&  \modules(e J e)\ar@/^1pc/[l]_{r}\ar@/_1pc/[l]_{\ell}\ar@{}[u]|{\bigcup}
}
\]
\smallskip

\noindent and does not require 
that $(Q,W)$ comes from 
a triangulation.

\subsection*{Koszul duality under Verdier quotient}
The second main result of the paper generalizes to $e\Gamma e$ the 
simple-projective duality for $\Gamma$, involving 
a bijection between finite hearts of $\pvd(\Gamma)$ and silting objects 
of $\per(\Gamma)$. 
Silting theory is a recent area 
of research that is receiving a 
lot of attention due to its connections with other areas of 
representation theory. 
If $\Lambda$ is a dg algebra, a silting object in $\per(\Lambda)$ 
is a direct 
sum $\mathbf{Y}=\oplus Y_i$ of non-isomorphic indecomposables that generates $\per(\Lambda)$. It 
admits a notion of silting 
mutation with respect to a summand, which is used to introduce 
the silting exchange graph $\SEG(\per(\Lambda))$. 
Originally, bijections between finite hearts and silting objects
have been studied for example by Keller--Nicol\'as, 
King-Qiu in \cite{kq}, Koenig-Yang in \cite{KoY}, to which we refer for contextualisation.  
Recent works in the context of non-positive dg algebras are 
\cite{SuYang}, where it is also related with 
a manifestation of Koszul duality, and \cite{fushimi}, to 
cite some examples. 
In the context of 3-Calabi-Yau algebras $\Gamma$ from decorated marked surfaces, the simple-projective duality, expressed as an isomorphism between (connected components of) $\EG(\pvd(\Gamma))$ and $\SEG(\per(\Gamma))$, is preserved by Verdier localization to $\pvd(e\Gamma e)$.
\begin{theorem}[{Theorem \ref{thm:SPdual}}]\label{thm_intro3}
    There exists an isomorphism of graphs
\[\EG^\bullet(\pvd(e\Gamma e)) \simeq \SEG^\bullet(\per(e\Gamma e)),\]
where $\SEG^\bullet(\per(e\Gamma e))$ is a full subgraph of the silting exchange graph of $\per(e\Gamma e)$ introduced in Definition \ref{def_seg_bullet}.
\end{theorem}
While the bijection between finite hearts and silting objects 
(Proposition \ref{le:dual}) does not require that $(Q,W)$ is dual to a triangulation, 
Theorem \ref{thm_intro3} means that under this additional assumption 
the correspondence is compatible with the notion of tilting 
and silting mutation.

\subsection*{Applications} We started the study of the categories 
associated with the dg algebra $e\Gamma e$ from the perspective of 
weighted decorated marked surfaces. This complementary perspective is 
considered in \cite{BMQS}, where we also compute (possibly part of) the 
space of Bridgeland stability conditions of $\pvd(e\Gamma e)$. 
The same categories and their bounded t-structures also appear 
in the definition of multi-scale stability conditions from 
\cite{BMS_compact}.
\par\medskip 
A weighted decorated marked surface (wDMS) is a bordered Riemann surface $\surf$ with 
a set of marked points $\MM$ on the boundary components and a map 
(weight) $\mathbf w:\Tri\to \Z_{\geq -1}$ defined on a finite set of internal 
points $\Tri\subset {\stackrel{\circ}{\surf}}$. One way of constructing a 
weighted marked surface is by contracting sub-surfaces of a decorated 
marked surface $\surf_{\mathbf 1}$. In this case we denote the resulting 
wDMS by $\colsur$. A mixed-angulation of a weighted marked surface is a 
tiling of $\colsur$ into polygons encircling exactly one point in $\Tri$ 
and whose number of edges is equal to the weight of such a point plus 
two. An edge of a mixed-angulation is a simple curve 
in $\surf\setminus\Tri$ connecting two marked points, up to isotopy.

In \cite{BMQS}, the category $\pvd(e \Gamma e)$ is realized as a 
triangulated 
category associated to a weighted decorated marked surface, 
obtained by \lq\lq collapsing\rq\rq\ subsurfaces of a marked 
bordered Riemann 
surface, together with a choice of a mixed-angulation. More 
precisely, $\Gamma$ comes from $\surf_{\mathbf 1}$ and the idempotent $e$ 
depends on the subsurface.
The exchange graph $\EGb(\pvd(e\Gamma e))$ is related with 
an exchange graph $\EGb(\colsur)$ of mixed-angulations 
of $\colsur$. (A union of connected components of)
the space of Bridgeland stability conditions 
on $\pvd(e\Gamma e)$, denoted $\stab^\bullet(\pvd(e\Gamma e))$, is 
shown to be isomorphic to a moduli space of quadratic 
differentials with high order poles and possibly high order 
zeroes, generalizing the Bridgeland-Smith correspondence 
of \cite{BS15} out of the usual $3$-Calabi-Yau setting.
\par\medskip
Putting together the main categorical result in \cite{BMQS}, 
the Koszul duality above, and the geometric model 
of silting objects in $\per(\Gamma)$ as arcs from \cite{Q3}, 
we obtain the following Theorem,  
which we interpret as a categorification of mixed-angulations of a weighted decorated marked surface.
\begin{theorem}\label{thm:app}
There are isomorphisms of exchange graphs
\[ \SEG^\bullet(\per(e\Gamma e)) \simeq \EGb(\pvd(e\Gamma e)) \simeq 
\EGb(\colsur),\]
where $\colsur$ is a weighted decorated marked surface 
corresponding to $e\Gamma e$.
\end{theorem}
In terms of arcs-objects correspondence, the geometric idea 
behind Theorem \ref{thm:app} 
is the following. Mixed-angulations of $\colsur$ can be deformed to 
partial triangulations of $\surf_{\mathbf 1}$. Arcs in triangulations 
of $\surf_{\mathbf 1}$ correspond to 
indecomposable summands of silting objects of $\per(\Gamma)$. 
Partial triangulations therefore corresponds to those partial 
silting objects of $\Gamma$, which are precisely the silting 
objects in $\SEGb(\per(e\Gamma e))$.

\subsection*{Acknowledgment} 
We would like to thank Lidia Angeleri, 
David Pauksztello, Jorge Vit{\'o}ria, Nicholas Williams, Dong Yang 
for useful discussions, Martin M{\"o}ller and Jeong-hoon So for a fruitful collaboration preceeding this article, and an anonymous referee for helping improving the exposition.

\section*{List of notation}\label{list_of_notation}\subsubsection*{Quivers and algebras}
\begin{compacthang}
    \item $\Lambda$\quad any differential graded (dg) algebra
\item $(Q,W)$ \quad a quiver with potential (Section \ref{subsec_qp})
\item $\Gamma=\Gamma(Q,W)$  differential
    non-positively graded Ginzburg algebra associated with 
    a quiver with potential $(Q,W)$ (Sections \ref{subsec_qp}, \ref{subset_ginzburg})
\item $J=H^0\Gamma$ \quad Jacobian algebra
\item  $I\sqcup I^c = Q_0$ \quad complementary subsets of the 
    set $Q_0$ of vertices of $Q$
\item $e=\sum_{j\not\in I}e_j$ \quad idempotent in $\Gamma$ or $J$
\item $(Q_I,W_I)$ \quad full subquiver of $(Q,W)$ with set of vertices $I$, not necessarily connected
\item $\Gamma_I:=\Gamma(Q_I,W_I)$ and $J_I:=H^0\Gamma_I$
\item $Q^{e J e}$ and $\ol{Q}^{e\Gamma e}$ \quad quivers of 
the algebras $e J e$ and $e\Gamma e$ respectively (Section \ref{sec_loc_pvd})
\end{compacthang}
\subsubsection*{Categories}
\begin{compacthang}
    \item $\cl D(\Lambda)$ \quad derived category of a dg algebra $\Lambda$
\item $\per\Lambda$ \quad perfect derived category of a 
    dg algebra $\Lambda$ (Section \ref{subset_ginzburg})
\item $\pvd\Lambda$ \quad perfectly valued derived category of a 
    dg algebra $\Lambda$ (Section \ref{subset_ginzburg})
\item $\modules(e J e) \simeq \modules J / \modules J_I$ \quad Serre quotient of abelian categories  (Section \ref{sec_loc_pvd})
\item $\pvd(e\Gamma e) \simeq \pvd\Gamma / \pvd \Gamma_I$ \quad Verdier quotient of triangulated categories  (Section \ref{sec_loc_pvd})
\item A \emph{finite heart} $\cl H$ is a heart of 
a bounded t-structure which is finite length and has finitely 
many simples up to isomorphism, $\Sim\cl H$  
(Section \ref{subsec_notation})
\item Fix a thick subcategory $\cl V$ of a triangulated 
category $\cl D$ and the Verdier quotient
\[\cl V \stackrel{\iota}{\longrightarrow}
\cl D \stackrel{\pi}{\longrightarrow} \cl D/ \cl V\]
\item A bounded t-structure in $\cl D$ is \emph{compatible with $\cl V$} if 
$\pi$ induces a bounded t-structure on $\cl D/\cl V$ (Definition \ref{def_induced_heart})
\item $\ol{\cl H}:= \pi(\cl H) \simeq \cl H / (\cl H\cap \cl V)$ 
for a $\cl V$-compatible heart $\cl H$ (Definition \ref{def_induced_heart})
\item $\ol{M}:=\pi(M)$ for an object $M\in \cl D$ (Notation \ref{notation_ess_im})
\item In Sections \ref{sec_tstructures} and \ref{sec_duality}\\ $\cl A=\modules J$, 
$\cl V= \pvd\Gamma_I$, $\cl K = \modules J_I= \modules J \cap \cl V$
\end{compacthang}
\subsubsection*{Exchange graphs}
\begin{compacthang}
\item $|G|$ \quad set of vertices of a graph $G$
\item $\EG(\pvd \Lambda)$ \quad exchange graph of bounded hearts 
    in $\pvd \Lambda$ (Definition \ref{def_EG1})
\item $\EGp(\pvd \Gamma)$ \quad the principal part 
of $\EG(\pvd \Gamma)$, i.e., 
    the connected component containing $\modules J$ as a 
    vertex. It only contains finite hearts 
    (Definition \ref{def_EG1} and Notation \ref{EG_Gamma})
\item $\EGp(\pvd \Gamma, \pvd \Gamma_I)$ \quad the full subgraph 
of $\EGp(\pvd \Gamma)$ whose vertices are hearts compatible 
with $\pvd \Gamma_I$ (Definition \ref{def_EG})
\item $\EGb(\pvd(e\Gamma e))$ \quad the full subgraph 
of $\EG(\pvd(e\Gamma e))$ 
whose vertices are hearts of type $\ol{\cl H}$ for some 
$\cl H\in\EGp(\pvd \Gamma)$ (Definition \ref{def_EG} and Notation \ref{EG_Gamma})
\item $|\EGb_{\cl K}(\pvd \Gamma)|$ \quad set of 
hearts $\cl H$ in 
$\EGp(\pvd \Gamma, \pvd \Gamma_I)$ satisfying 
$\ol{\cl H}\in\EGb(\pvd(e\Gamma e))$ and $\cl H\supset \cl K$ (Equation \ref{Jin_map})
\item $\SEG(\per \Lambda)$ \quad silting exchange graph 
of $\per \Lambda$ (Definition \ref{def_SEG})
\item $\SEGp(\per \Gamma)$ \quad principal part 
of $\SEG(\per \Gamma)$ 
containing $\Gamma$ as a vertex, it is isomorphic to 
$\EGp(\pvd\Gamma)$ under simple-projective duality 
for $\Gamma$ (Definition \ref{def_SEG})
\item $\pSEGV(\per \Gamma)$ \quad exchange graph of $\cl V$-perpendicular 
partial silting objects in $\per \Gamma$ (Definition~\ref{def_pSEG})
\item $\pSEGVb(\per \Gamma)$ \quad principal part 
of $\pSEGV(\per\Gamma)$ consisting on partial silting objects 
that can be completed to silting objects in~$\SEGp(\per \Gamma)$ (Definition~\ref{def_seg_bullet})
\item $\SEGb(\per (e\Gamma e))$ \quad silting exchange graph 
of $\per(e\Gamma e)$ of silting objects corresponding to 
partial silting objects in~$\pSEGVb(\per \Gamma)$ 
(Definition~\ref{def_seg_bullet})

\end{compacthang}

\section{Preliminaries on categories and notation}
\label{sec_prelim}

\subsection{Categories and t-structures: notation}\label{subsec_notation}

We fix the notations and some basic facts regarding abelian and triangulated 
categories, their bounded t-structures, and the tilting procedure.
\par\smallskip
Through the paper, $\acf$ is an algebraically closed field, all categories 
are $\acf$-linear, and all subcategories are full. The shift functor of a 
triangulated category is denoted by $[1]$, and distinguished triangles are 
represented as $X\to E\to Y$, likewise short exact sequences. For an 
additive (either abelian or triangulated) category $\calC$ with 
subcategories $\calH_1,\calH_2$ and for a set of objects $\cl B$, we define
\begin{align*}
    \bra \cl B\ket&:=\{M\in\calC\mid \exists\ T,F\in\mathrm{Add}\, \cl B \text{ s.t. } T\to M\to F\},~\mbox{and} \\
    \calH_1*\calH_2&:=\{M\in\calC\mid \exists\ T\in\calH_1, F\in\calH_2 \text{ s.t. }T\to M\to F \}.
\end{align*}
If $\calH_1,\calH_2$ satisfy $\Hom(\calH_1,\calH_2)=0$, then we write
$\calH_1\perp\calH_2$ for $\calH_1*\calH_2$. If 
$\calH_1\perp\calH_2 = \calC$, the pair $(\calH_1,\calH_2)$ is called a torsion pair. 
We denote moreover by $\calB^{\perp_\calC}$ and ${^{\perp_\calC}}\calB$ the 
right and left orthogonal to $\calB$ in $\cl C$, omitting the subscript when 
there is no confusion. More precisely,
\[
    \calB^{\perp_\calC}\,:=\,\left\{C\in\calC:\Hom_\calC(B,C)=0,\ \forall\,B\in \calB\right\},
\]
and ${^{\perp_\calC}}\calB$ is defined similarly.
If $\calC$ is triangulated, we denote by $\thick(\cl B)$ the smallest
thick additive full subcategory in $\calC$ containing $\cB$. 
It is a triangulated subcategory.
\begin{rmk*}
A finite length abelian category will be said to be \emph{finite} if it has finitely
many simple objects, up to isomorphism. The set of isomorphism classes 
of simples in an 
abelian category $\cl A$ will be denoted $\Sim\cl A$.
\end{rmk*}
A subcategory $\cl S$ of an abelian category $\cA$ is called
a \emph{Serre subcategory} if it is abelian and for any short exact sequence
\[
0 \to A_1 \to E \to A_2 \to 0\]
in $\cA$, we have that $E \in \cl S$ if and only if $A_1,A_2 \in \cl S$. 
\\
Serre subcategories of a finite lenght abelian category $\calA$ are in 
bijection with subsets of the set $\Sim\calA$. 
\par\medskip
A \emph{t-structure} on a triangulated category $\calD$ is
the torsion part $\cl P$ of a torsion pair 
(so that $\calD=\cl P \perp \cl P^\perp$) satisfying $\cl P[1]\subset 
\cl P$. The $t$-structure is said to be \emph{bounded} 
if $\calD=\cup_{m\in\Z}\cl P[m]\cap\cl P^\perp[-m]$.
The \emph{heart} of a t-structure $\cl P\subset\calD$ is the full 
subcategory $\cl H=\cl P\cap\cl P^\perp[1]$, which is abelian. 
If $\cl H\subset\calD$ is the heart of a bounded t-structure, we call 
it a \emph{(bounded) heart}.
The bounded t-structure and its heart determine each other uniquely 
and hence we will use them interchangeably. The cohomology 
functor $\calD\to \cl H$ is denoted by $H^0_{\cl H}$ and the subscript 
is omitted if there is no possibility of confusion. As standard notation, 
we write $H^n_{\cl H}$ for $H^0_{\cl H}\circ[n]$.
\par
If the heart of a bounded t-structure is a finite length abelian category 
with finitely many simples, we call it a \emph{finite heart}. 
\par\medskip
Given a torsion pair $(\calT,\calF)$ in the abelian heart of a bounded t-structure $\cl H= \calT\perp\calF$,
there is a new heart $\mu_{\torsionfree}^\sharp\h:=\torsionfree[1] \perp_\calD \torsion$. The operation of creating the new heart is 
known as Happel-Reiten-Smal\o\ (HRS) tilt or 
\emph{forward tilt with respect to $(\torsion,\torsionfree)$} of $\h$,
see e.g.~\cite{tiltingbook}. 
The inverse operation is the backward tilt with respect to the torsion 
class; the new heart is denoted by $\mu^\flat_\torsion\h:=
\torsionfree\perp_\calD\torsion[-1]$.
\par
A forward tilt $\h\to\mu_{\torsionfree}^{\sharp}\h$ (resp., backward 
tilt $\h\to\mu_{\torsion}^{\flat}\h$) is \emph{simple} if the corresponding 
torsion free class $\torsionfree$ (resp., torsion class $\torsion$) is 
generated by a simple $S$ of $\h$, i.e. $\torsionfree=\<S\>$ (resp., $\torsion=\bra S\ket$).
For a finite heart~$\h$ with a rigid simple~$S$, i.e., 
such that~$\Ext^1(S,S)=0$, the simple forward tilt with respect
to~$S$ exists and is denoted by $\mu_{S}^{\sharp}\h$. It is finite too. 
While in general it is not true that the forward tilt of a finite heart 
is finite, if a torsion-free class contains finitely many indecomposables, 
then tilting at $\torsionfree$ is equivalent to perform a finite sequence 
of simple tilts and the new heart is finite.
\par\medskip
We recall general definitions about 
graphs and we define two  
graphs encoding information about t-structures of a 
triangulated category.
\begin{definition}
    Let $\operatorname{G}=(V,E)$ be a (oriented) graph with 
    set of vertices $V$ and (directed) edges $E$. A \emph{subgraph} 
    of $\operatorname{G}$ consists of $(V',E')$ 
    where $V'\subseteq V$ and $E'\subseteq E$ such that the endpoints of edges in $E'$ connect vertices in $V'$. 
    Given $V'\subset V$, the subgraph $(V',E')$ is \emph{full} 
    if $E'\subset E$ is the maximal subset of edges 
    connecting vertices in $V'$;  it is 
    \emph{connected} if, given any two vertices, there is a path 
    connecting them. 
\end{definition}
\begin{definition}\label{def_EG1} Let $\calD$ be a triangulated category and $\calA$ a fixed distinguished finite bounded heart of $\calD$.
    \begin{itemize}
\item The \emph{exchange graph} $\EG(\calD)$ is the oriented graph whose vertices are all hearts in $\calD$
and whose directed edges correspond to simple forward tilt between 
them. 
\item The \emph{principal component} $\EGp(\calD)$ is the 
\emph{connected component} of the exchange graph 
$\EG(\calD)$ containing the distinguished heart $\cl A$.
\end{itemize}
\end{definition}

\subsection{Quivers with potential}\label{subsec_qp}
Recall that a quiver with potential $(Q,W)$ is defined by a set of 
vertices $Q_0$, a set of arrows~$\alpha:i\to j\in Q_1$, the source and 
target functions $s(\alpha)=i$, $t(\alpha)=j$, and a formal sum of 
cycles~$W$ in the completion $\widehat{\mathbf{k} Q}$ of the path algebra of $Q$ with 
respect to the bilateral ideal generated by the 
arrows. We refer to \cite{DWZ} for basic 
notions. We assume that a quiver $Q$ has no loops nor 2-cycles. 
For any $k\in Q_0$, we denote by $\mu_k$ the operation of mutation 
of $(Q,W)$ at the vertex $k$, producing another quiver with 
potential $(Q',W')$ with no loops and no 2-cycles in the following 
way. The set of vertices $Q_0'$ is the same as $Q_0$, 
arrows $\alpha$ incoming from or outgoing to the vertex $k$ are 
replaced by their opposite $\alpha^*$, and for any pair of 
consecutive arrows $\alpha:i\to k$ and $\beta:k\to j$ passing 
through $k$, an arrow $[\alpha\beta]:i\to j$ is added. The 
potential $W'$ is defined 
as $\widetilde{W}+\sum_{\alpha,\beta}[\alpha\beta]\beta^*\alpha^*$, 
where $\widetilde{W}$ is obtained from $W$ by replacing any 
composition $\alpha\beta:i\to k\to j$ by $[\alpha\beta]$. We say 
that $(Q,W)$ is non-degenerate if moreover any finite sequence of 
mutations produces a quiver with no loops nor 2-cycles.
\par
We say that a quiver is \emph{finite} if it has finitely many vertices 
and finitely many arrows.
\par\medskip
The Jacobian algebra of $(Q,W)$ is the quotient $J(Q,W)$ of the 
complete path algebra $\widehat{\acf Q}$ by the bilateral
ideal $\del W=\bra\del_\alpha W\mid \alpha\in Q_1\ket$, 
where $\del_\alpha W$ is obtained by cyclically permuting $W$ in a 
way that $\alpha$ is in the first position in each word, and then 
deleting $\alpha$ and all words not containing $\alpha$.
\par \medskip
We make the assumption on $(Q,W)$ that it is non-degenerate and that $J(Q,W)$ is finite dimensional.
\par\medskip
For a fixed finite non-degenerate quiver with potential $(Q,W)$, we denote by $\Gamma(Q,W)$ the complete Ginzburg differential graded (dg) algebra of Calabi-Yau dimension 3. It is the dg path algebra of a graded quiver $\ol{Q}$ obtained from $(Q,W)$ in the following way
\begin{itemize}
    \item $\ol{Q}_0=Q_0$,
    \item if $\alpha\in Q_1$, then $\alpha \in \ol{Q}_1$ in degree $0$,
    \item for any $\alpha\in Q_1$, there is $\alpha^\vee \in \ol{Q}_1$ in degree $-1$,
    \item for any $k\in\ol{Q}_0$, there is a loop $\ell_k:k\to k\in \ol{Q}_1$ in degree $-2$. 
\end{itemize}
The differential $d$ is the unique continuous linear endomorphism, homogeneous of degree 1,
which satisfies the Leibniz rule and takes values $d\alpha^\vee=\del_\alpha W$, $d\ell_k=\sum_{\alpha\in Q_1}\ell_k[\alpha,\alpha^\vee]\ell_k$.
We refer to \cite{KY,keller_dilog} for more details. 
The zero-th homology of this dg algebra is the Jacobian algebra 
of the original quiver with potential: $H^0\Gamma(Q,W)\simeq J(Q,W)$. 

\subsection{Localization of categories and hearts} \label{subsec_quotients}

We recall the notions of quotient of an abelian category by a Serre subcategory and of the Verdier quotient of a triangulated category, and give some definitions concerning localization of bounded t-structures. References for these notions are Neeman's book \cite{neeman}, in particular Appendix A and Chapter 2, and \cite{antieau}.
\par
Given a Serre subcategory $\cl S$ of an abelian category $\mathcal{H}$, one can construct the quotient $\mathcal{H}/\cl S$, together with a projection functor $\pi :\mathcal{H}\to \mathcal{H}/\mathcal{S}$ defined in the following way.
\begin{definition} If $\cl S$ is a Serre subcategory of an abelian category $\mathcal{H}$, $\mathcal{H}/\mathcal{S}$ is a category with the same set of objects as $\mathcal{H}$ and where,
 for all $X,Y\in Obj\left(\mathcal{H/S}\right)$, a morphism in $\Hom_{\mathcal{H/S}}(X,Y)$ is an equivalence class of roofs $(\tilde\eta,\eta)$ of the form
\[
	X \stackrel{\tilde\eta}{\leftarrow} Z \stackrel{\eta}{\rightarrow} Y\]
for some $Z\in Obj(\mathcal{H/S})$, $\tilde\eta\in\Hom_{\mathcal{H}}(Z,X)$, $\eta\in\Hom_{\mathcal{H}}(Z,Y)$, with $\ker\tilde\eta\in\mathcal{S}\ni \coker\tilde\eta$, and with obvious composition law
\[\xymatrix@R=1pt{&& T \ar[ld]\ar[rd]\\
 & Z \ar[dl]\ar[dr] &  & U \ar[dl]\ar[dr] &  \\
X &   & Y &   & W\ .
}
\]
(Here we say that two roofs are equivalent, if they are dominated by a third
roof.) The projection $\pi:\mathcal{H}\to \mathcal{H}/\cl S$ is a functor whose kernel
is~$\mathcal{S}$, sending objects identically in themselves and morphisms
\[\Hom_{\mathcal{H}}(A,B)\ni f \mapsto (id,f)\in \Hom_{\mathcal{H/S}}(A,B).\]
\end{definition}
\begin{lemma}[{\cite[{Lemma A.2.3}]{neeman}, \cite{gabriel}}] The category $\mathcal{H}/\cl S$ is an abelian category. The functor $\pi$ is exact, and takes the objects of $\cl S$ to objects in $\mathcal{H}/\cl S$ isomorphic to zero. Furthermore, $\pi$ is universal with this property, and the subcategory~$\cl S$ is the full subcategory of objects of $\mathcal{H}$ whose images under $\pi$ are isomorphic to zero.
\end{lemma}
Roughly speaking, this means that all morphisms in $\cl S$ have become invertible, and objects in $\cl S$ can be treated as zero objects, and that we have a short exact sequence
\[0 \to \cl S \stackrel{\iota}{\rightarrow} \mathcal{H} \stackrel{\pi}{\rightarrow}
\mathcal{H}/\cl S \to 0,\]
with $\iota$ and $\pi$ exact functors of abelian categories.
\par
\medskip
\paragraph{\textbf{Verdier (triangulated) localization}}
Given a triangulated category $\calD$ and a triangulated subcategory
$\calV\hookrightarrow \calD$, we can construct the so-called
\emph{Verdier quotient}
$\calD/\calV$ with a procedure similar to that discussed for abelian
categories. The category $\calD/\calV$ has the same objects as $\calD$ and
a morphism between two objects $R$ and $S$ in $\calD/\calV$ is an
equivalence class of roofs
\[R\stackrel{\tilde f}{\leftarrow} T \stackrel{f}{\rightarrow} S
\]
with $\tilde f$ a morphism in $\calD$ fitting in a distinguished triangle
$T \stackrel{\tilde f}{\rightarrow} R \rightarrow V$, with $V$ in $\calV$.
There is a natural triangulated functor $\pi:\calD\to \calD/\calV$, called
the \emph{Verdier localization}. If the triangulated subcategory $\calV$ is
thick, that is it contains all direct summands of its objects, then
\[0 \to \calV \stackrel{\iota}{\to} \calD \stackrel{\pi}{\to} \calD/\calV \to 0
\]
is a short exact sequence of categories with exact
functors (\cite[Proposition~2.3.1]{V}).

\begin{rmk*}\label{notation_ess_im} Whenever $\pi:\calD\to \calD/\cl V$ is a quotient 
    functor of triangulated categories, and $\cl B$ is a subcategory 
    of $\calD$, by $\pi(\cl B)$ we will mean the \emph{essential image} 
    of $\cl B$ through~$\pi$. This will apply in particular to the 
    image of an abelian heart $\h\subset\calD$.
\end{rmk*}
\begin{rmk*} To avoid confusion, if $M\in \calD$ is an object, we 
will denote by $\ol{M}$ the same object regarded in the quotient 
category $\calD/\calV$.
\end{rmk*}

We recall a general fact about bounded t-structures under localization.
\par
\begin{prop}[{\cite[Proposition~2.20]{antieau}}]\label{prop_antieau}
Let $\iota:\cl C \to \cl D$ be a t-exact fully faithful functor of triangulated
 categories equipped with bounded t-structures,
with a well-defined quotient functor $\pi:\cl D \to \cl D/ \cl C$.
Let $\h_\D$ and $\h_\calC$ be the two hearts in $\D$ and $\calC$ respectively.
Then the following are equivalent
	\begin{itemize}
		\item[a)]the essential image $\iota(\h_\calC)\subset \h_\D$ is a Serre subcategory, and
		\item[b)] the quotient $\calD/\cl C$ has a bounded t-structure such that $\pi$ is t-exact, whose
		heart is equivalent to $\h_\D/\h_\calC$.
	\end{itemize}
\end{prop}
The (bounded) t-structure corresponding to
$\h_\D/\h_\calC$ in $\calD/\cl C$ of point b) is described in
{\cite[Proposition~2.20]{antieau}}.
\par\medskip
Suppose $\cl V$ is a full thick triangulated subcategory of a triangulated 
category $\calD$, and let $\calH$ be a bounded heart of $\calD$.
\begin{definition}[{\cite[Definition 2.3]{BMS_compact}}]\label{def_induced_heart} 
We say that $\calH$ is \emph{compatible} with $\cl V$ 
if $\calH\cap \cl V$ is a bounded heart 
of $\cl V$ making the inclusion of $\cl V$ in $\calD$ t-exact, and it is a Serre full subcategory of $\calH$. 
\par
If a heart in $\calD/\cl V$ arises as in Proposition \ref{prop_antieau},
we say that it is \emph{induced} by the heart $\cl H_\calD$, 
and we call it a 
heart \emph{of quotient type}. It will be denoted by $\ol{\cl H}$ whenever 
it is equivalent to $\cl H/(\cl H\cap \calV)$.
\end{definition}
\par
We introduce moreover the following notation, based on Definition \ref{def_EG1}:
\begin{definition}\label{def_EG} Fix a distinguished bounded t-structure of $\calD$ with finite heart $\cl A$ 
compatible with $\calV$.
\begin{itemize}
\item $\EGp(\calD,\cl V)$ is the full subgraph of $\EGp(\calD)$ 
whose vertices are finite hearts of bounded t-structures in 
$\calD$ that are compatible with $\cl V$;
\item $\EG^\bullet(\calD/\calV)$ is the full subgraph of 
$\EG(\calD/\cl V)$ whose vertices can be 
realized as quotients of hearts $\cl H\in\EGp(\calD)$.
\end{itemize}
\end{definition}
\par
The next lemma is easy to understand, and is proven 
in \cite[Proposition 6.7]{BMQS} for the dual notion of HRS 
backward tilt (at a torsion class).
\begin{lem}\label{quot_tilted_hearts}
  Suppose that $\calH,\calH'$ are $\cl V$-compatible hearts of~$\calD$. Then 
	\[\cl H/(\cl H\cap\calV) \simeq \cl H'/(\cl H'\cap \calV)\subset \calD/\calV,\]
if $\calH'=\mu^\sharp_{\torsionfree}\calH$ at some torsion free class $\torsionfree\subset \calH$ such that $\torsionfree\subset \calV$.
\end{lem}
We end by noting that if $\cl H_{\cl C}$ and $\cl H_{\calD}$ are finite hearts as in Proposition \ref{prop_antieau}, then $\iota\torsionfree$ is a torsion-free class in $\cl H_{\cl D}$ if and only if $\torsionfree$ is a torsion-free class in $\cl H_{\cl C}$. 
\par

\subsection{Recollements of triangulated categories and hearts}\label{subset_recoll}
If the inclusion and quotient functor $\iota$ and $\pi$ are exact and admit left and right adjoints, we are in a \emph{recollement} situation. Let $\cl A, \cl B, \cl C$ be either abelian or triangulated categories. 
\begin{definition}
    A recollement $R(\cl A, \cl B, \cl C)$ is a diagram
    \begin{equation}\label{recollement1}
    \xymatrix{
    0 \ar[r] & \calA \ar[r]^{\iota} & \cl B \ar[r]^{\pi} \ar@/_1pc/[l]_{q} \ar@/^1pc/[l]_{p} & \cl C \ar[r] \ar@/_1pc/[l]_{\ell} \ar@/^1pc/[l]_{r} &0
    }
    \end{equation}
such that $(\ell,\pi,r)$ and $(q,\iota,p)$ are adjoint triples, and, moreover, the functors $\iota, \ell, r$ are fully faithful with $\Im\iota=\ker\pi$.
\par
We write $R_{ab}(\cl A, \cl B, \cl C)$ for a recollement of abelian categories and $R_{tr}(\cl A, \cl B, \cl C)$ for a recollement of triangulated categories. 
\end{definition}
A recollement $R_{ab}(\cl H', \cl H, \cl H'')$ is a recollement of hearts $R_{\heart}(\cl H', \cl H, \cl H'')$ if it is induced by a t-exact recollement of triangulated categories $R_{tr}(\calD', \calD, \calD'')$ by composing with the cohomology functors and the inclusions, in the sense that it fits into a commutative diagram
\begin{equation}\label{recoll_hearts1}
    \xymatrix@C=2cm{\calD' \ar[r]^{\iota} \ar[dd]^{H^0_{\cl H'}}
    & \calD \ar[r]^{\pi}\ar@/^1pc/[l]_{p}\ar@/_1pc/[l]_{q} \ar[dd]^{H^0_{\cl H}}
    &  \calD''\ar@/^1pc/[l]_{r}\ar@/_1pc/[l]_{\ell} \ar[dd]^{H^0_{\cl H''}}\\ \\
    \cl H' \ar[r]^{\iota'}\ar@/^1pc/[uu]^{\epsilon'} & \cl H \ar[r]^{\pi'}\ar@/^1pc/[l]_{p'}\ar@/_1pc/[l]_{q'} \ar@/^1pc/[uu]^{\epsilon} &  \cl H'' \ar@/^1pc/[l]_{r'}\ar@/_1pc/[l]_{\ell'} \ar@/^1pc/[uu]^{\epsilon''},
    }
    \end{equation}
where the $\epsilon$'s map denote the inclusions and the $H^0$'s maps denote the cohomology functors defined by the bounded t-structures. In a recollement of hearts,
\[\begin{aligned}
    q'&=H^0_{\cl H'}\circ q \circ \epsilon & \ell'&=H^0_{\cl H}\circ \ell \circ \epsilon'' \\
    \iota'&=H^0_{\cl H}\circ \iota \circ \epsilon' & \pi'&=H^0_{\cl H''}\circ \pi \circ \epsilon \\
    p'&=H^0_{\cl H'}\circ p \circ \epsilon & r'&=H^0_{\cl H}\circ r \circ \epsilon''.
\end{aligned}
\]

The notions of compatible and glued torsion pairs defined
below are taken from \cite{psarou}. They ensure that a recollement of hearts from a recollement of triangulated categories induces another recollement of hearts by tilting. Proposition \ref{glue_torsion_psa} characterizes torsion pairs
whose images
under the localization map $\pi$ are still torsion pairs.
\begin{definition}\label{def_compatible_torsion} Let $R_{ab}(\cl A', \cl A, \cl A'')$
be the recollement of abelian categories
\begin{equation}\label{abrecoll}\xymatrix@C=2cm{
    \cl A' \ar[r]^{\iota} & \cl A \ar[r]^{\pi}\ar@/^1pc/[l]_{p}\ar@/_1pc/[l]_{q}  &  \cl A'' \ar@/^1pc/[l]_{r}\ar@/_1pc/[l]_{\ell}.
}\end{equation}
A torsion pair $(\cl X, \cl Y)$ in $\cl A$ is said to be 
\emph{compatible} with $R_{ab}(\cl A', \cl A, \cl A'')$ if 
$\ell \pi(\cl X)\subset \cl X$ or $r\pi(\cl Y)\subset \cl Y$.
\end{definition}
\begin{prop}[{\cite[Proposition 9.1]{psarou}}]\label{glue_torsion_psa} Let $R_{ab}(\cl A', \cl A, \cl A'')$ be a recollement of abelian categories.
\begin{itemize}
    \item[i)] Let $(\cl X, \cl Y)$ be a torsion pair in $\cl B$. Then
    \begin{itemize}
    \item the pair $(q(\cl X), p(\cl Y))$ is a torsion pair in $\cl A'$, and
    \item the pair $(\pi(\cl X), \pi(\cl Y))$ is a torsion pair in $\cl A''$ if and only if $\ell \pi(\cl X)\subset \cl X$ if and only if $r \pi (\cl Y)\subset \cl Y$.
    \end{itemize}
    \item[ii)] Let $(\cl X', \cl Y')$ and $(\cl X'',\cl Y'')$ be torsion pairs in $\cl A'$ and $\cl A''$ respectively. Then
\[\begin{aligned}
        \cl X &:= \left\{A\in \cl A \mid q(A)\in \cl X', \pi(A)\in\cl X''\right\},~\mbox{and}\\
        \cl Y &:= \left\{A\in \cl A \mid p(A)\in\cl Y', \pi(A)\in\cl Y''\right\}
\end{aligned}\]
         is a torsion pair in $\cl A$ compatible with $R_{ab}(\cl A', \cl A, \cl A'')$.
    \end{itemize}
\end{prop}

\begin{definition}The torsion pair $(\cl X, \cl Y)$ in Proposition
\ref{glue_torsion_psa} (ii), is said to be \emph{glued} from $(\cl X',\cl Y')$
and $(\cl X'',\cl Y'')$. Similarly, if $\cl A', \cl A, \cl A''$ are hearts of bounded t-structures, the t-structure obtained by (forward or backward) tilting $\cl A$ at the glued torsion pair will be said to be \emph{glued} from bounded t-structures with hearts $\cl A', \cl A''$.
\end{definition}
Let $R_{tr}(\calD', \calD,\calD'')$ be a recollement of triangulated categories. It is shown 
in \cite[Section 9]{psarou} that, 
if $R_{\heart}(\cl A', \cl A, \cl A'')$ is a recollement of hearts and the torsion 
pair $(\cl X, \cl Y)$ in $\calA$ is glued from $(\cl X', \cl Y')$ and $(\cl X'', \cl Y'')$, 
then 
\[R_{\heart}\left(\mu^\sharp_{\cl Y'}\cl A', \mu^\sharp_{\cl Y}\cl A, \mu^\sharp_{\cl Y''}\cl A''\right)\quad\text{and}\quad R_{\heart}\left(\mu^\flat_{\cl X'}\cl A', \mu^\flat_{\cl X}\cl A, \mu^\flat_{\cl X''}\cl A''\right)\]
are other recollements of hearts. 
It is however not true in general that adjoint 
triples restrict to subcategories.

On the other hand, whether a recollement of abelian hearts lifts to a recollement of triangulated categories 
is also an interesting question considered for instance in \cite{lidia}.
\par
In Section \ref{sec_loc_pvd} we will be interested in an 
intermediate situation where a localization of triangulated categories 
\[0 \to \cl C'\stackrel{\iota}{\to} \cl C \stackrel{\pi}{\to}\cl C''\to 0
    \]
does not admit adjoints, but some bounded hearts fit into a recollement, 
and we would like to understand when tilting bounded t-structures 
in $\cl C''$ arise from tilting a bounded t-structure in $\cl C$. 
We will show in Theorem \ref{tilt_ses} that under  
suitable conditions, we can \emph{\lq\lq glue\rq\rq} bounded t-structures 
even in the absence of a recollement of triangulated categories.

\subsection{Categories from quivers with potential}\label{subset_ginzburg}

In this paper we work with the following triangulated categories associated 
to a differential graded (dg) algebra $\Lambda$:
    \begin{itemize}
    \item the derived category $\D(\Lambda)$,
    \item the perfect derived category $\per(\Lambda)\subset\D(\Lambda)$, and
    \item the perfectly valued derived 
    category $\pvd(\Lambda)\subset\D(\Lambda)$. 
    \end{itemize}
The latter consists of objects $M$ in $\D(\Lambda)$ such 
that $\Hom_{\calD(\Lambda)}(P, M[i])$ has finite dimension for 
any $P\in\per(\Lambda)$ and for any $i\in\mathbb{Z}$.
\par\medskip
Here and through the rest of the paper, we fix a 
non-degenerate finite quiver with potential $(Q,W)$ whose 
Jacobian algebra is finite dimensional. We denote 
the 3-Calabi-Yau Ginzburg dg algebra $\Gamma(Q,W)$ simply by $\Gamma$ and its Jacobian algebra $J(Q,W)$ 
by $J$. We recall that $J=H^0\Gamma$. 
The main features of $\pvd(\Gamma)$ that will be 
used in the next Sections are recalled here. We refer the 
reader to \cite{KY} for details.
\par\medskip
The category  
$\pvd(\Gamma)$ coincides with the
subcategory of $\D(\Gamma)$ consisting of dg modules 
of finite-dimensional total co-homology. It is 
generated by the simple dg modules $S_i$ 
associated with the vertices of $Q$. 
It is $\Hom$-finite and of Calabi-Yau 
dimension $3$, meaning that there is a natural 
isomorphism of $\acf$-vector spaces
\begin{equation}\label{cy}
    \nu:\Hom_{\pvd\Gamma}(E,F)\stackrel{\sim}{\to}
\Hom_{\pvd\Gamma}(F,E[3])^\vee
\end{equation}
for any pair of objects $E,F$.  
Indeed, for any pair of simple dg modules $S_i, S_j$, 
$i,j\in Q_0$, a basis of $\Hom^{d+1}(S_i,S_j)$ is given by the 
(finitely many) 
degree $(-d)$ arrows $i\stackrel{-d}{\longrightarrow}j$ 
in $\ol{Q}$. 
The Calabi-Yau property \eqref{cy} can be written as 
\[\Hom^{n}(S_i,S_j) \simeq \Hom^{3-n}(S_j,S_i)^\vee,
\]
and is combinatorically expressed by the 
existence of pairs of opposite arrows in $\ol{Q}$ in 
degree $-d,-d^*$ such 
that $(d+1)+(d^*+1)=3$, with the convention 
that $\Hom^0(S_i,S_i)=\acf$ is implicititely assumed and not 
represented by arrows.
\par\medskip
Since $\Gamma=\Gamma(Q,W)$ is homologically smooth 
(defined and proven in \cite[{Section 7}]{KY}), we have the following sequence of inclusions
\[
    \pvd(\Gamma)\subset\per(\Gamma)\subset\D(\Gamma),
\]
where the category $\per\Gamma$ is
generated by the indecomposable projective dg 
modules~$P_k=e_k\Gamma$, where $e_k$ is the idempotent 
element in $\Gamma$ associated with the vertex $k\in Q_0$.
As~$\Gamma=\oplus_{k=1}^n P_k$, we have 
\begin{equation}\label{PSrelation}
\Hom_{\per(\Gamma)}(P_k, S_i)=\delta_{ki} \acf, 
\end{equation}
for any $k,i\in Q_0$.
\par\medskip
We recall that $\pvd(\Gamma),\per(\Gamma),\D(\Gamma)$ 
are equivalent to $\pvd(\Gamma'),\per(\Gamma'),\D(\Gamma')$, 
if~$\Gamma'=\Gamma(Q',W')$ is associated with a 
quiver $(Q',W')$ obtained from $(Q,W)$ by mutations at its 
vertices. 
\par
The standard t-structure with heart $\Modules(J)$ on 
$\calD(\Gamma)$ restricts to bounded t-structures with 
heart $\rep(Q,W)\simeq \modules J$ on $\pvd\Gamma$ and 
$\per\Gamma$, that we still call standard. The abelian 
category $\modules J$ is finite-length with simple objects 
in bijection with the vertices of $Q$. The correspondence 
between HRS simple tilts in $\Modules J\subset\calD(\Gamma)$ 
or $\modules J\subset\pvd(\Gamma)$ and mutations at a 
vertex $i\in Q_0$ realizes $\rep(Q',W')=\modules J'$ as 
another heart in $\pvd(\Gamma)$, when $(Q',W')$ is obtained 
from $(Q,W)$ by a sequence of mutations. In facts, if 
$\EGp(\pvd\Gamma)$ is the connected component of the exchange graph 
containing $\modules J$ as a vertex, any other vertex corresponds 
to a finite heart equivalent to $\modules J'$ for the Jacobian algebra $J'$ 
of a quiver with potential $(Q',W')$ related to $(Q,W)$ by 
a sequence of mutations.

For later purposes, we recall here the relations between simples of $\modules J$ and simples of $\mu^{\sharp/\flat}_S \modules J$ inside $\pvd(\Gamma)$, assuming that $(Q,W)$ has no loops nor 2-cycles. Let $\cl H=\modules J$. Let $\Sim\cl H=\left\{[S_j]\,|\, j=1,\dots, n\right\}$ and $S=S_i$ for some $i$.

Then we have that
\[\Sim \mu_S^\sharp\cl H = \{\left[S[1]\right],[F_j]\,|\, j\neq i\}  
\quad \text{and} \quad 
\Sim \mu_S^\flat\cl H = \{\left[S[-1]\right],[E_j]\,|\, j\neq i\}\]
where $F_j$ and $E_j$ are defined as cones of universal extensions in $\pvd(\Gamma)$ and fit into the following distinguished triangles
\[
S \otimes \Ext^1_{\cl H}(S_j,S) \to F_j \to S_j
\quad\text{and}\quad
S_j \to E_j \to S\otimes \Ext^1_{\cl H}(S,S_j).
\]
In particular $F_j=S_j$ if $\Ext^1_{\cl H}(S_j,S)=0$, 
and $E_j=S_j$ if $\Ext^1_{\cl H}(S,S_j)=0$.

\section{Localization of Calabi-Yau categories from quivers}\label{sec_loc_pvd}

We use the hypotheses and the notation set in 
Section \ref{subset_ginzburg}. Moreover, we fix once for all
a proper subset $I\subset Q_0$ of the set of vertices of $Q$. 
We denote by $I^c=Q_0\setminus I$ its complement. The quiver 
with potential $(Q_I,W_I)$ 
is the full subquiver of $(Q,W)$ whose set of vertices is $I$, and with 
potential $W_I$ obtained from $W$ by deleting all cycles 
passing through 
vertices of $I^c$. We write $\Gamma_I$ for $\Gamma(Q_I,W_I)$ 
and
$J_I=H^0\Gamma_I=\acf Q_I/\del W_I$. The quiver $(Q_I,W_I)$ 
and the associated algebras $\Gamma_I,J_I$ satisfy the same 
properties as $(Q,W)$, $\Gamma$ and $J$. The following Lemma 
is clear and will be useful later
\begin{lem}
    Let $k$ be a vertex of $Q$ in $I\subset Q_0$, 
    or $k\not\in I$ be a vertex of $Q$ such that there are 
    no arrows from~$i$ to~$k$ or from~$k$ to~$i$ for 
    all~$i\in I$. Then $(\mu_k(Q,W))_I = (Q_I,W_I)$.
\end{lem}
We introduce here the Verdier quotients 
$\cl D(\Gamma)/\cl D(\Gamma_I)$ and  
$\pvd(\Gamma)/\pvd(\Gamma_I)$ and their underlying algebra  
$e\Gamma e$ for an idempotent $e\in\Gamma$. 
Under our 
assumptions, the Verdier quotient is a dg quotient in the 
sense of \cite{K1, D}, so the two perspectives are available. 
The goal of this Section is to present and study 
some homological properties 
of $\pvd(e\Gamma e)=\pvd(\Gamma)/\pvd(\Gamma_I)$.

\subsection{Idempotents}
For any $k\in Q_0$, we let $e_k$ be the idempotent element 
in $\Gamma$ associated with the vertex $k$ of the quiver $Q$. 
We define \[e:=\sum_{j\in I^c}e_j\]
to be the idempotent associated
to the
complement $I^c=Q_0\setminus I$.  We use the same notation, 
$e_k$ and $e$, for the same elements regarded as idempotents 
in $J$. As explained in~\cite{kalckyang2},
the Jacobian algebra $J_I$ and the Ginzburg algebra
$\Gamma_I$ of the subquiver $(Q_I,W_I)$ are isomorphic to 
the quotients by the bilateral ideals generated by $e$:
\[
    J_I \simeq J/ J e J, \text{ and }
    \Gamma_I \simeq \Gamma/\Gamma e \Gamma 
    \text{, respectively,}
\]
and we have exact inclusions of categories
\[\Modules (J_I) \hookrightarrow \Modules(J), \text{ and }
\modules(J) \hookrightarrow \modules(J),\]
and 
\[\calD(\Gamma_I)\hookrightarrow \calD(\Gamma).\]
At the abelian level, $\modules J_I$ is a Serre subcategory 
of $\modules J$ generated by the simple objects labelled by 
indexes in $I\subset Q_0$. In fact, Serre subcategories 
of $\modules J$ are in bijection with full 
subquivers $(Q_I,W_I)$ of $(Q,W)$, or equivalently with 
subsets $I\subset Q_0$. 
\par\medskip
The other algebras we are interested in here are $eJe$ 
and $e\Gamma e$, that we describe briefly. Let $A$ denote 
either $J=J(Q,W)$ or $\Gamma=\Gamma(Q,W)$, then $eAe$ consists of all (graded) 
paths in $A$ that start and end at some vertex in $I^c$. 
For $A=J$, this is still a quiver algebra for a finite 
quiver with potential $(Q^{eJe},W^{eJe})$ with
\begin{itemize}
    \item set of vertices $Q^{eJe}_0=I^c$, and
    \item set of arrows $Q^{eJe}_1$, consisting of arrows 
    in $Q_1$ connecting two vertices in $I^c$, and a basis 
    of arrows $j_1\to j_2$ for distinct paths 
    connecting $j_1,j_2\in I^c$ in $\acf Q$, not in the bilateral 
    ideal $\del W_I$.
\end{itemize}
If $J$ is a finite dimensional algebra, $eJe$ is still a 
finite dimensional 
algebra, and the quiver with potential $(Q^{eJe},W^{eJe})$ 
has finitely many vertices and finitely many arrows. See 
Figure \ref{fig} for examples.
\begin{figure}\centering

\begin{tikzpicture}[
arrow/.style={->,>=stealth,thick},
 midlabel/.style={midway,fill=white}]

\path (-4.5,0) node (q)  {$A_3=$};
\path (-3.5,0) node (x)  {$\bullet_1$};
\path (-2.5,0) node (y)  {$\bullet_2$};
\path (-1.5,0) node (z)  {$\bullet_3$};
\draw[arrow] 
(x) edge node[midway, above] {\tiny{$\alpha$}} (y)
(y) edge node[midway, above] {\tiny{$\beta$}} (z);


    \path (1,0) node (q2)  {$\mu_2A_3=$};
\path (2,0) node (xx)  {$\bullet_1$};
\path (3,-1) node (yy)  {$\bullet_2$};
\path (4,0) node (zz)  {$\bullet_3$};
\path (6,0) node (ww)  {$W=[\alpha\beta]\beta^*\alpha^*$};
\draw[arrow] 
(xx) edge node[midway, above] {\tiny{$[\alpha\beta]$}} (zz)
(zz) edge node[midway, below right] {\tiny{$\beta^*$}} (yy)
(yy) edge node[midway, below left] {\tiny{$\alpha^*$}} (xx);

\path (-1.5,-2) node (qbar)  {$Q^{eJ e}=$};
\path (-.5,-2) node (xxbar)  {$\bullet_{\bar{1}}$};
\path (1.5,-2) node (zzbar)  {$\bullet_{\bar{3}}$,};
\draw[arrow] 
(xxbar) edge node[midway, above] {\tiny{$\alpha\beta$}} (zzbar);

\end{tikzpicture}
\caption{The choice of $I=\{2\}$ for the quivers  with potentail $(A_3,0)$ 
and $(\mu_2 A_3,W)$ produces the same quiver $Q^{eJe}$ with 
trivial potential. The abelian quotient category does not 
change. 
}\label{fig}
\end{figure}

\par\smallskip
A similar argument shows that $e\Gamma e$ is the dg path 
algebra of 
a graded quiver $\ol{Q}^{e\Gamma e}$ with finitely many 
vertices 
$I^c$, but possibly infinitely many arrows: for any pair of 
vertices $j_1,j_2\in I^c$ and for any degree $-d\in\Z_{\leq 0}$, 
an arrow 
$j_1\stackrel{-d}{\to}j_2$ appears whenever there is a path of 
degree $-d$ in $\Gamma(Q,W)$. See Figure \ref{fig2} for an example. The differential is the same 
of $\Gamma$ (defined in section \ref{subsec_qp}), inducing an isomorphism $H^k(e\Gamma e) \simeq e H^k(\Gamma) e$.
\par
Writing $e\Gamma e= e\Gamma \otimes \Gamma e$ identifies the 
dg 
algebra $e\Gamma e$ with the endomorphism dg algebra of the 
projective module $\Gamma e = \sum_{j\in I^c}\Gamma e_j$ 
in $\calD(\Gamma)$.

\begin{figure}\centering
\begin{tikzpicture}[
arrow/.style={->,>=stealth,thick},
 midlabel/.style={midway,fill=white}]
    every node/.style={circle, draw, minimum size=1cm},
    every loop/.style={looseness=8}
\node (Q) at (-5,0) {$\ol{Q}=$};
\node (v1) at (-4,0) {$\bullet_1$};
\node (v2) at (-2,0) {$\bullet_2$};
\node (QQ) at (1,0) {$\ol{Q}^{e\Gamma e}=$};
\node (v) at (2,0) {$\bullet_{\bar{1}}$};
\draw[arrow]
(v1) .. controls (-3.7, -.8) and (-4.3,-.8) .. node[below] {\tiny{$-2$}} node[left] {\tiny{$\ell_1$}} (v1);
\draw[arrow]
(v2) .. controls (-1.7, -.8) and (-2.3,-.8) ..node[below] {\tiny{$-2$}} node[left] {\tiny{$\ell_2$}} (v2);
\draw[arrow]
(v1) .. controls (-3.5,.6) and (-2.5,.6) .. node[below] {\tiny{$0$}} node[above] {\tiny{$\alpha$}} (v2);
\draw[arrow]
(v2) .. controls (-2.5,-.6) and (-3.5,-.6) .. node[below] {\tiny{$-1$}} node[above] {\tiny{$\alpha^\vee$}} (v1);
\draw[arrow]
(v) .. controls (3, .6) and (3,-.6) ..  node[right] {\tiny{$-2k-1,\ k\geq 0$}} (v);
\draw[arrow]
(v)  .. controls (2.3, -.8) and (1.7,-.8) ..  node[below] {\tiny{$-2$}} (v);
\end{tikzpicture}
\caption{The graded quiver $\ol{Q}^{e\Gamma e}$ obtained from $\Gamma=
\Gamma(A_2,0)$ for $e=e_2$ has infinitely many arrows: one in degree 
$-2$ and one for any odd negative integer, obtained by 
concatenating $\alpha$, $(\ell_2)^k$, and $\alpha^\vee$ in 
$\ol{Q}=\ol{A_2}$. The differential is the same of $\Gamma$.
}\label{fig2}
\end{figure}

\subsection{Recollements induced by idempotents}
As explained in~\cite{kalckyang2},
the Ginzburg algebra
$\Gamma_I$ is isomorphic to $\Gamma/\Gamma e \Gamma$,
for $e=\sum_{j\in I^c}e_j$ defined above. On the other hand, the Verdier
quotient $\calD(\Gamma)/\calD(\Gamma_I)$ coincides with
$\calD(e\Gamma e)$. 
 The derived
category associated to $\Gamma$ and its restriction to $I$
fits in a
recollement of triangulated categories induced by an idempotent,
\begin{equation}\label{tr_recollement}
\xymatrix{
0 \ar[r] & \calD(\Gamma_I)
\ar[r]^{\iota} & \calD(\Gamma) \ar[r]^{\pi} \ar@/_1pc/[l]_{q} \ar@/^1pc/[l]_{p} & \calD(e\Gamma e) \ar[r] \ar@/_1pc/[l]_{\ell} \ar@/^1pc/[l]_{r} &0
}
\end{equation}
where
\[\begin{aligned}
   &q(-)= - \stackrel{\mathbb{L}}{\otimes}_\Gamma \Gamma/ \Gamma e \Gamma, && \ell(-)= -\stackrel{\mathbb{L}}{\otimes}_{e\Gamma e} e\Gamma,\\
   &\iota(-)=\mathbb{R}\Hom_{\Gamma/\Gamma e \Gamma}(\Gamma/\Gamma e\Gamma, -), && \pi(-)=\mathbb{R}\Hom_\Gamma(e\Gamma, -),\\
 &p(-)= \mathbb{R}\Hom_\Gamma(\Gamma/\Gamma e \Gamma,-), && r(-)=\mathbb{R}\Hom_{e\Gamma e}(\Gamma e, -),
\end{aligned}\]
are adjoint pairs of fully faithful functors. 
They induce a recollement of hearts

\begin{equation}\label{recoll_hearts}
\xymatrix@C=2cm{\calD(\Gamma_I) \ar[r]^{\iota} \ar[dd]^{H^0_{\Mod J_I}}
& \calD(\Gamma) \ar[r]^{\pi}\ar@/^1pc/[l]_{p}\ar@/_1pc/[l]_{q} \ar[dd]^{H^0_{\Mod J}}
&  \calD(e\Gamma e)\ar@/^1pc/[l]_{r}\ar@/_1pc/[l]_{l} \ar[dd]^{H^0_{\Mod e J e}}\\ \\
\Mod J_I \ar[r]^{\iota}\ar@/^1pc/[uu]^{\epsilon_I} & \Mod J \ar[r]^{\pi}\ar@/^1pc/[l]_{p}\ar@/_1pc/[l]_{q} \ar@/^1pc/[uu]^{\epsilon} &  \Mod e J e, \ar@/^1pc/[l]_{r}\ar@/_1pc/[l]_{\ell} \ar@/^1pc/[uu]^{\epsilon_e}
}
\end{equation}
where the $\epsilon$'s maps are inclusions, and the functors between hearts are denoted exactly as the triangulated functors with abuse of notation. 
We refer to \cite{psarou} for 
examples of recollements induced by idempotents and for the general theory. See also \cite[\S7]{kalckyang2} for the
special example involving Calabi-Yau algebras from quivers with potential.
\par\medskip
At the level of finitely generated modules over the Jacobian
algebras, $\modules  J_I \subset \modules  J$ is localizing and
co-localizing, and we have a similar recollement of abelian categories
\begin{equation}\label{mod_recoll}
    \xymatrix@C=2cm{
    \modules J_I \ar[r]^{\iota} & \modules J \ar[r]^{\pi}\ar@/^1pc/[l]_{p}\ar@/_1pc/[l]_{q}  &  \modules e J e \ar@/^1pc/[l]_{r}\ar@/_1pc/[l]_{\ell},
    }
\end{equation}
where we use the same notation as in \eqref{recoll_hearts} for obvious reasons:
\[\begin{aligned}
    &\iota(-) = \Hom_{ J/ J e  J}( J/ J e J,-),&& \pi(-)=\Hom_{ J}(e J,-),\\
    &q(-)= - \otimes_{ J}  J/  J e  J, && \ell(-)= -\otimes_{e J e} e J,\\
    &p(-)= \Hom_{ J}( J/ J e  J,-), && r(-)=\Hom_{e J e}( J e, -).
\end{aligned}\]

The article \cite{geigle} by Geigle and Lenzing is a reference text for the description of $\modules J_I$ and $\modules eJe$ inside $\modules J$.

A trivial corollary of Proposition \ref{glue_torsion_psa} 
is that:
\begin{cor} Any torsion pair in $\modules(e  J e)$ lifts to a torsion
    pair in $\modules  J$ (in non unique way).
\end{cor}

\subsection{The quotient category $\pvd(e\Gamma e)$}
By definition, the maps $\iota$ and $\pi$ appearing in \eqref{tr_recollement} restrict to the perfectly valued derived categories of the same dg algebras, i.e., to modules with finite dimensional total homology. Therefore the Verdier quotient $\pvd(\Gamma)/\pvd(\Gamma_I)$
is equivalent to $\pvd(e\Gamma e)$
\[\xymatrix{
    0 \ar[r]& \pvd(\Gamma_I) \ar[r]^{\iota}
    &\pvd(\Gamma) \ar[r]^{\pi} & \pvd(e \Gamma e)\ar[r]& 0. }
    \]
The functors $\ell$ and $q$ restrict to the perfect categories
\[\xymatrix{
    0 & \per(\Gamma_I) \ar[l]
    &\per(\Gamma)\ar[l]_{q} & \per(e \Gamma e)\ar[l]_{\ell}& 0. \ar[l] }
    \]
\begin{rmk}The left adjoints $q,\ell$ do not descend to the $\pvd$ level. 
 \end{rmk}
We learned this argument from Dong Yang; we include a
proof for completeness. The crucial 
feature is Calabi-Yau dimensional finiteness.
\begin{proof}Suppose that $\pi:\pvd(\Gamma) \to \pvd(e\Gamma e)$ had a 
    left adjoint $\pi_!$. Then the pair
    \[(\cl X, \cl Y) := \left(\pi_!(\pvd(e\Gamma e)), 
    \iota(\pvd(\Gamma/ \Gamma e\Gamma))\right)\]
would be a stable t-structure (i.e., the aisle and the co-aisle 
are closed under arbitrary shifts) on $\pvd(\Gamma)$, hence a split 
t-structure since $\pvd(\Gamma)$ is Calabi-Yau. 
Indeed, the Calabi-Yau condition applied 
to $X\in \cl X$ and $Y \in \cl Y$,
\[\Hom(Y,X)^\vee \simeq \Hom(X,Y[3]) = 0,\] 
implies that, 
for any $E\in \pvd(\Gamma)$,
\[x(E) \to E \to y(E) \to x(E)[1], \quad x(E)\in \cl X,\ y(E)\in \cl Y\]
is a splitting sequence with $y(E)\stackrel{0}{\to}x(E)[1]$ and $E \simeq x(E)\oplus y(E)$.

Similarly, one proves that $\pi$ does not admit a right
adjoint, and that there are no adjoints in $\pvd(\Gamma)$
to the inclusion $\iota$.
\end{proof}

\begin{theorem}\label{rmk_Hrestrict}
The cohomology theories in the derived categories of \eqref{recoll_hearts} 
associated with $\Mod J_I$, $\Mod J$, and $\Mod e  J e$
restrict to bounded t-structures and cohomology theories on
$\pvd(\Gamma_I)$, $\pvd(\Gamma)$, and $\pvd(e \Gamma e)$, with hearts
$\modules  J_I$, $\modules J$, and $\modules e J e$, respectively.
\end{theorem}
\begin{proof}
The statement follows from  \cite[{Proposition 3.2}]{kalckyang2}, being $ J$, $ J_I$, $e J e$ Noetherian rings, and
$\Gamma$, $\Gamma_I$, $e\Gamma e$ non-positive dg algebras.
\end{proof}

In particular, 
\[\modules J_I=\pvd\Gamma_I\cap\modules J\subset \modules J\]
as a Serre subcategory,  
and $\modules eJe\simeq \modules J / \modules J_I$.
\par
\begin{rmk*}\label{EG_Gamma}    We fix the standard heart $\cl A:=\modules J$ as a distinguished heart 
for $\pvd(\Gamma)$, so that $\EGp(\pvd\Gamma)$ is the connected 
component of $\EG(\pvd\Gamma)$ containing $\modules J$ as a vertex. \par
We recall that the graph $\EGb(\pvd(e\Gamma e))$ 
is the full subgraph 
of $\EG(\pvd(e\Gamma e))$ whose vertices correspond to hearts 
$\ol{\cl H}$ for some $\cl H$ in $\EGp(\pvd\Gamma)$. 
Then $\EGb(\pvd(e\Gamma e))$ contains $\modules(eJe)$ as a vertex.
\end{rmk*}
\par
While homological smoothness guarantees the 
inclusion $\pvd(\Lambda)\subset\per(\Lambda)\subset\calD(\Lambda)$, for $\Lambda=\Gamma,\Gamma_I$, we have no reasons to expect the 
same chain of inclusions for $e\Gamma e$, and the relation 
between $\pvd(e\Gamma e)$ and $\per(e\Gamma e)$ remains unknown to us,
\[
\xymatrix@C=2cm{
\calD(\Gamma_I) \ar[r]^{\iota} 
& \calD(\Gamma) \ar[r]^{\pi}\ar@/^1pc/[l]_{p}\ar@/_1pc/[l]_{q} 
&  \calD(e\Gamma e)\ar@/^1pc/[l]_{r}\ar@/_1pc/[l]_{\ell} \\
\per(\Gamma_I) \ar@{}[u]|{\bigcup}& \per(\Gamma) \ar@{}[u]|{\bigcup} \ar[l]_{q} & \per(e\Gamma e) \ar[l]_{\ell}\\
\pvd(\Gamma_I) \ar@{}[u]|{\bigcup} \ar[r]^{\iota} & \pvd(\Gamma) \ar[r]^{\pi} \ar@{}[u]|{\bigcup} & \pvd(e\Gamma e). 
}
\]
On the other hand, we can compare 
$\per (e\Gamma e)$ and $\pvd(\Gamma_I)$, as they can be realized 
inside $\per \Gamma$, by composing $\iota$ with the inclusion $\pvd\Gamma\subset\per\Gamma$
\[\label{orth_inclusion}
\xymatrix@C=2cm{
\pvd(\Gamma_I) \ar[r]^{\iota} &  \per \Gamma  & \per (e\Gamma e) \ar[l]_{\ell}.
}
\]
\begin{lemma}\label{lem:perpV} 
The orthogonality relation
$\ell(\per(e \Gamma e))=
{^{\perp_{\per\Gamma}}}\iota(\pvd\Gamma_I)$ holds in $\per\Gamma$.
\end{lemma}

\begin{proof}
By equation \eqref{tr_recollement},
$\ell(\calD(e\Gamma e))={^{\perp_{\calD(\Gamma)}}}
\iota(\calD(\Gamma_I))$
(see \cite[{\S 1.4.3.3}]{f-pervers}), and this, together with
\eqref{orth_inclusion}, implies that
$\ell(\per (e\Gamma e)) \subseteq {^{\perp_{\per \Gamma}}}\iota(\pvd\Gamma_I)$.
To prove the statement, we need to show that this inclusion is indeed an
equality. For simplicity, we identify the categories associated with $\Gamma_I$
and with $e\Gamma e$ with their images under $\iota$ and $\ell$ respectively.
We show that, for any $X\in\per\Gamma$ not in $\per (e\Gamma e)$, 
then $X\not\in {^{\perp_{\per\Gamma}}\pvd\Gamma_I}$.
\par
By \cite[{Lemma 2.14}]{plamondon}, any object $X$ in $\per\Gamma$ 
(resp.\ in $\per (e\Gamma e)$) admits a minimal
perfect presentation, i.e., it is quasi-isomorphic to a finite direct sum of
shifts of projective objects $\bigoplus_{k\in K\subset Q_0} P_k[d_k]$ (resp.\
of $\bigoplus_{k\in K\subset I^c} P_k[d_k]$), and its differential, as a degree~1 map,
is a strictly upper triangular matrix with entries in the ideal 
of~$\Gamma$
generated by the arrows of the double quiver~$\oQ$. Note that for any 
arrow $\alpha\colon i\to j$
in $\oQ$ that induces some irreducible morphism
$P_j\xrightarrow{f_\alpha} P_i[\deg \alpha]$, we have $\Hom(f_\alpha,S)\equiv 0$ 
for any
simple $S\in\modules J$ (since either $\Hom(P_j,S)=0$ or $\Hom(P_i[d_i],S)=0$).
This implies, in particular, that for a simple $S\in \modules J$,
\begin{equation}\label{eq:XS}
	\Hom_{\per\Gamma}(X,S) \=
    \bigoplus_k\Hom_{\per\Gamma}(P_k[d_k], S).
\end{equation}
Therefore, for any
$X$ in $\per\Gamma$ not in $\per (e\Gamma e)$, its perfect presentation has a summand
$P_i[d_i]$ for some $i\in I$, hence by \eqref{PSrelation}
$\Hom^0(X,S_i[-d_i])\neq 0$ with $S_i[-d_i]\in\pvd \Gamma_I$. 
\end{proof}
Lemma \ref{lem:perpV} will be used for the proof of the 
simple-projective duality for $\pvd e\Gamma e$ 
and $\per e \Gamma e$, in Section \ref{sec_duality}.

\subsection{Gluing torsion pairs, and tilting in $\pvd(e\Gamma e)$}

The perfectly valued categories 
\[\pvd(\Gamma_I) \to \pvd(\Gamma) \to \pvd(e\Gamma e),\] 
despite admitting no adjoints to the 
quotient functor, are sandwiched between recollements 
$R_{tr}\big(\cl D(\Gamma_I), \cl D(\Gamma), \cl D(e\Gamma e)\big)$ 
and $R_{ab}\big(\modules(J_I), \modules(J), \modules(eJe)\big)$ 
of triangulated and abelian categories, involving the derived 
categories of the Ginzburg algebras and the small module 
categories of the Jacobian algebras, respectively. See 
equations \eqref{tr_recollement} and \eqref{mod_recoll}. 
We exploit this feature to show that tilting the standard 
heart $\modules(e J e)$ of $\pvd(e\Gamma e)$ at any torsion 
pair results in another heart of quotient type. To make 
the presentation ligther, in the following we 
write denote $\cl H_\torsionfree^\sharp$ 
for $\mu^\sharp_\torsionfree\cl H$ and $\cl H_\torsion^\flat$ for 
$\mu^\flat_\torsion\cl H$, whenever $(\torsion,\torsionfree)$ is a 
torsion pair in a bounded heart $\cl H$.
\par
\begin{prop}\label{tilt_ses}
    Let $(\cl U, \cl V)$ and $(\torsion,\torsionfree)$ be torsion pairs in 
    the standard hearts $\modules J_I$ and $\modules e J e$
    of $\pvd\Gamma_I$ and $\pvd e\Gamma e$ respectively. 
    Then $(\cl U, \cl V)$ and $(\torsion, \torsionfree)$ \emph{glue} to a 
    torsion pair $(\cl X, \cl Y)$ in $\modules J$ which is 
    compatible with the recollement of abelian categories. 
    The (forward or backward) tilts at $(\cl U, \cl V)$, $(\cl X,\cl Y)$, 
    $(\torsion,\torsionfree)$ produces new short exact sequences
    \[ 0 \to \left(\modules  J_I \right)^\sharp_{\cl V} 
    \stackrel{\iota}{\rightarrow} \left( \modules J\right)^\sharp_{\cl Y} 
    \stackrel{\pi}{\rightarrow} \left(\modules e J e \right)^\sharp_\torsionfree \to 0
    \]
    and
    \[ 0 \to \left(\modules  J_I \right)^\flat_{\cl U} 
    \stackrel{\iota}{\rightarrow} \left( \modules J\right)^\flat_{\cl X} 
    \stackrel{\pi}{\rightarrow} \left(\modules e J e \right)^\flat_\torsion \to 0
    \]
    of hearts of bounded t-structures in $\pvd\Gamma_I \stackrel{\iota}\to \pvd\Gamma\stackrel{\pi}\to \pvd(e\Gamma e)$.
    \end{prop}
    \begin{proof}  
    We prove the proposition for backward tilt, the forward tilt being completely analogous. 
    The closure by limits of $\torsion\subset \modules e J e$ (resp.\  
    $\cl U\subset\modules J_I$) is a torsion class 
    $\widetilde\torsion$ in $\Mod eJ e$ (resp.\ 
    $\widetilde{\cl U}$ in $\Mod J_I$)
    such that $\widetilde\torsion\cap\modules e J e= \torsion$ (resp.\ 
    $\widetilde{\cl U}\cap\modules J_I= \cl U$), see \cite{crawley1994locally}. 
    We consider the recollement of abelian hearts \eqref{recoll_hearts}. 
    The gluing method by Beilinson-Bernstein-Deligne, together with 
    \cite[{Propositions 9.1(ii) and 10.2}]{psarou} 
    implies that the torsion classes $\widetilde\torsion$ 
    and $\widetilde{\cl U}$ induce 
    a torsion class $\widetilde{\cl X}$ in $\Mod J$ such that 
    $\pi(\widetilde{\cl X})=\widetilde\torsion$, 
    $q(\widetilde{\cl X})=\widetilde{\cl U}$, and a new 
    recollement of abelian hearts 
    \[\xymatrix{
    \left(\Mod J_I\right)^\flat_{\widetilde{\cl U}} \ar[r]^{\iota''} 
    & \left(\Mod J\right)^\flat_{\widetilde{\cl X}} \ar[r]^{\pi''}\ar@/^2pc/[l]^{p''}\ar@/_2pc/[l]_{q''} 
    &  \left(\Mod e J e\right)^\flat_{\widetilde{\torsion}} \ar@/^2pc/[l]^{r''}\ar@/_2pc/[l]_{\ell''}
    }.
    \]
    According to \cite{psarou} (cf, proof of Proposition 10.2 and 
    equations therein), we write
    \begin{align}
    \left(\Mod J_I\right)^\flat_{\widetilde{\cl U}} = 
    \left\{X\in \calD(\Gamma_I)\mid \right. & H^0_{\Mod J_I}(X)\in \widetilde{\cl U}, H^{-1}_{\Mod J_I}(X)\in \widetilde{\cl U}^\perp,\notag\\
    &\left. H^i_{\Mod J_I}(X)=0,\ i\neq 0,-1\right\}, \label{e1}
    \\
    \left(\Mod J\right)^\flat_{\widetilde{\cl X}} = 
    \left\{X\in \calD(\Gamma)\mid\right. &H^0_{\Mod J}(X)\in \cl X, H^{-1}_{\Mod J}(X)\in \widetilde{\cl X}^\perp, \notag \\ 
    &\left. H^i_{\Mod eJ e}(X)=0,\ i\neq 0,-1\right\},~\mbox{and}\label{e2}\\
    \left(\Mod eJ e\right)^\flat_{\widetilde\torsion} = 
    \left\{X\in \calD(\Gamma)\mid\right. &H^0_{\Mod eJ e}(X)\in \widetilde\torsion, H^{-1}_{\Mod eJ e}(X)\in \widetilde\torsion^\perp, \notag\\ 
    &\left. H^i_{\Mod eJ e}(X)=0,\ i\neq 0,-1\right\}.\label{e3}
    \end{align}
    Since the cohomology functors appearing in 
    \eqref{e1}, \eqref{e2}, and \eqref{e3} 
    restrict to the perfectly valued categories by Theorem \ref{rmk_Hrestrict}, then 
    \begin{gather*}
    \left(\Mod J_I\right)^\flat_{\widetilde{\cl U}} \cap \pvd \Gamma_I = \left(\modules J_I\right)^\flat_{\cl U}, \\
    \left(\Mod J\right)^\flat_{\widetilde{\cl X}} \cap \pvd \Gamma = \left(\modules J\right)^\flat_{\cl X}, \text{ and}\\
    \left(\Mod eJ e\right)^\flat_{\widetilde\torsion} \cap \pvd e\Gamma e = \left(\modules eJ e \right)^\flat_{\torsion},
    \end{gather*} 
     where $\cl X = \widetilde{\cl X}\cap\modules J$. For the thesis to follow, 
     we need to know that $\iota''$ and $\pi''$ descend to the 
     intersections. This is true because they are the restrictions of 
     $\iota$ and $\pi$ to $\pvd \Gamma_I$ and $\pvd\Gamma$ respectively, see 
     \eqref{recoll_hearts}. Boundedness of the resulting t-structures 
     follows from HRS tilting theory and 
     \cite[Proposition\,2.20]{antieau}.
    \end{proof}
It follows that tilting at a torsion pair a heart 
$\ol{\cl H}=\cl H / \cl H\cap \cl V$ of $\pvd (e\Gamma e)$ 
from $\cl H\in\EGp(\pvd\Gamma)$, 
always produces another heart of quotient type.
\begin{cor}\label{any_heart_quot} 
Let $\ol{\cl H} \in \EGb(\pvd (e\Gamma e))$ and 
${\ol{\cl H}}^\sharp_f$ be the heart of another bounded 
t-structure in $\pvd(e\Gamma e)$ obtained by tilting $\ol{\cl H}$ at a 
torsion free class $f$. Then ${\ol{\cl H}}^\sharp_f$ 
is of quotient type. \\
More precisely, if 
$\ol{\cl H}=\cl H/\left(\cl H\cap\pvd\Gamma_I\right)$ for 
some $\cl H\in|\EGp(\pvd \Gamma)|$, then there exists 
$\torsionfree\subset \cl H$ such that 
${\ol{\cl H}}^\sharp_f = \cl H' /\left(\cl H' \cap\pvd\Gamma_I\right)$ 
for $\cl H'= \mu^\sharp_\torsionfree\cl H$.
\end{cor}
\begin{proof}
    The statement follows from Proposition \ref{tilt_ses} , 
    using the fact that any bounded heart in $\EGb(\pvd(e\Gamma e))$ 
    is $\modules (e' J' e')\subset \pvd(e'\Gamma'e')$ 
    with $\pvd(e'\Gamma'e')\simeq \pvd(e\Gamma e)$, 
    for $\Gamma'=\Gamma(Q',W')$, $J'=H^0\Gamma'$, and the quivers 
    $(Q',W')$ and $(Q,W)$ related by mutations.
\end{proof}

\section{Decorated marked surfaces and collapsing surfaces}
\label{sec_DMS}
We recall some definitions related with decorated marked surfaces 
and the construction of the dual quiver of a triangulation, and 
with the operation of collapsing subsurfaces. In 
Sections \ref{sec_tstructures} and \ref{sec_duality} we will 
restrict to a Ginzburg algebra $\Gamma$ of a quiver  with potential 
from a decorated marked surface. 
\subsection{Decorated marked surfaces and exchange graphs}\label{subsec_cate_from_surf}
A \emph{decorated marked surface (DMS)} $\surfo$ \emph{(without 
puntures)} is determined by 
the following data:
\begin{itemize}
  \item a compact genus $g\geq 0$ surface $\surf$,
  \item a set $\mathbf{M}\ne\emptyset$ of \emph{marked points} on the 
  boundary $\partial\surf$ of $\surf$, and
  \item a set $\Tri$ of \emph{decorations}, i.e., points in the 
  interior $\surf^\circ$ of $\surf$,
\end{itemize}
such that each boundary component of $\partial\surf$ contains at 
least one marked point, and
\begin{equation}\label{eq:aleph}
  |\Tri|=4g+2|\partial\surfo|+|\mathbf{M}|-4\ge1,
\end{equation}
where $|\partial\surfo|$ is the number of boundary components.
\par 
An open arc on $\surfo$ is an isotopy class of a curve whose 
interior is in $\surf^\circ\setminus\Tri$ and whose endpoints 
are marked points. 
\par
A \emph{(decorated) triangulation} $\TT$ of $\surfo$ is a 
collection of open 
arcs that, together with boundary segments defined by 
consecutive 
marked points, subdivides $\surfo$ into once-decorated 
triangles, i.e., triangles encircling exactly one decoration. 
The open arcs of $\TT$ are usually called 
\emph{internal edges} of the triangulation, and once-decorated 
triangles are said \emph{internal triangles}. 
The assumptions on $\surfo$ ensure that no decorated 
triangulation admits \emph{self-folded triangles}.
There is an operation of forward flip on triangulations, 
which changes an arc of a triangulation 
by moving its endpoints 
along the boundary of the quadrilateral containing it, 
without crossing $\Tri$. We refer 
to \cite{LabaFragQP1,LabaFragQP2} for 
the classical notion of flip of edges of a triangulation and 
to \cite{qiubraid16} for the decorated version. We denote by 
$\EG(\surfo)$ the exchange graph of triangulations 
of $\surfo$,
whose vertices are triangulations and whose directed edges  are 
forward flips.
\par\medskip
Given a triangulation $\TT$ of $\surfo$, there is an 
associated quiver with potential $(Q_\TT,W_\TT)$,
where
\begin{itemize}
  \item the vertices of $Q_\TT$ are the open arcs in $\TT$,
  \item the arrows of $Q_\TT$ are the (anticlockwise) angle 
  between open arcs in $\TT$, and
  \item $W_\TT$ is the sum of all 3-cycles $W_\TT$, that are 
  the compositions of three open arcs for any triangle 
  of $\TT$ whose edges are all internal edges.
\end{itemize}
\begin{definition}We say that a quiver with potential is 
    of \emph{geometric type}, precisely when it can be 
    defined by this procedure.\end{definition}
Fixing an initial triangulation $\TT_0$,
we let $\EGp(\surfo)$ be the principal component 
of $\EG(\surfo)$ 
containing $\TT_0$. We let moreover $(Q_{\TT_0},W_{\TT_0})$ 
be the quiver 
associated with $\TT_0$, and $\Gamma:=\Gamma_{\TT_0}$ be its Ginzburg 
algebra. 
Key features of $\pvd\left(\Gamma_{\TT_0}\right)$ are the 
correspondences between
\begin{itemize}
    \item simple objects of a heart and internal edges of a 
    triangulation of $\surfo$, and
    \item simple (forward) tilts and (forward) flips.
\end{itemize}
\par

This correspondence was first studied by Labardini Fragoso 
in \cite{LabaFragQP1} for the unpuctured case 
and in \cite{LabaFragQP2} in the presence of \emph{self-folded} 
triangles, and can be expressed more precisely in the 
following theorem including results contained 
in \cite{qiubraid16, Q3}. 

\begin{theorem}\label{thm:QQ}
There is an isomorphism of oriented 
graphs $\EGp(\surfo)\cong\EGp(\pvd\Gamma)$,
sending a triangulation $\TT$ to a heart $\h_\TT$, 
such that 
simples of $\h_\TT$ are naturally parameterized by 
open arcs in $\TT$ or vertices of the associated 
quiver $Q_\TT$. Each directed edge of $\EGp(\surfo)$ corresponding to flipping an internal edge 
of the triangulation matches with a directed edge of $\EGp(\pvd\Gamma)$ corresponding to 
a simple forward tilt of a heart of a bounded t-structure 
in $\pvd(\Gamma)$.
Moreover, the dimension of $\Ext^1$ between simples 
of $\h_\TT$ equals the number of arrows between the 
corresponding vertices in $Q_\TT$.
\end{theorem}
In the setup of Theorem \ref{thm:QQ}, if 
$R$ and $S$ are the simples corresponding to open 
arcs $\beta$ and $\gamma$ in a triangulation $\TT$, 
we have the following easy lemma.
\begin{lemma}[No-arrow/$\Ext^1$-vanishing Condition]
    \label{lem:no arrow}
If $\beta$, $\gamma$ are not counter-clockwise consecutive 
internal edges of an internal triangle of $\TT$, i.e., if 
there is no arrow in $(Q_\TT,W_\TT)$ from $\beta$ to $\gamma$, 
then $\Ext^1(R, S)=0$.
\end{lemma}

\subsection{Collapsing subsurfaces}\label{subsec_collapse}
Let $\Gamma=\Gamma(Q_{\TT_0},W_{\TT_0})$ be the Ginzburg 
algebra of 
the dual quiver of a triangulation $\TT_0$ of a decorated marked 
surface (without punctures) $\surfo$ 
and $e=\sum_{j\in I^c}e_j$ an 
idempotent as in the previous sections. 
The category $\pvd(e\Gamma e)$ admits a combinatorial 
description in terms of a \emph{weighted decorated 
marked surface} obtained by \emph{collapsing a subsurface 
of $\surfo$} roughtly associated with the subset $I$. 
This perspective is considered in \cite{BMQS} and 
briefly summarized here.
\par
Let $\Sigma$ be a subsurface 
cut out of $\surf$ along simple 
closed curves of $\surf\setminus\Tri$. It can be realized as a 
decorated marked surface $\Sigma_\Tri$ by taking 
its closure $\ol\Sigma$ and adding marked points on 
its (new) boundary components. 
For a fixed triangulation $\TT$ of $\surfo$, assume that we 
cut $\Sigma_\Tri$ in a way that the restriction of $\TT$ 
to $\Sigma_\Tri$ identifies a subquiver $(Q_I,W_I)$ 
of $(Q_\TT,W_\TT)$.
The details of this procedure 
can be found in \cite[\S~5.1]{BMQS}.
Then one can \emph{collapse $\Sigma_\Tri$ in $\surfo$} to 
get a \emph{weighted decorated marked surface (wDMS)} $\colsur$, 
defined as follows and diagrammatically represented by
\be
\begin{tikzcd}
    \Sigma_\Tri \ar[r,hookrightarrow] & 
    \surfo \ar[r,rightsquigarrow] & \colsur\, .
\end{tikzcd}
\ee
\begin{figure}[ht]\centering
	\makebox[\textwidth][c]{
		\begin{tikzpicture}[scale=.5]
			\draw[thick,fill=cyan!10](-4,-4)to[out=180,in=-90](-7,-1)to[out=90,in=180](-3,4)to[out=0,in=180](2,2)to[out=0,in=180](5,3)to[out=0,in=90](7,1)to[out=-90,in=0](3,-3)to[out=180,in=0](0,-2)to[out=180,in=0](-4,-4);
			\draw[thick](-5.5,-1)to[bend right=60](-4,-1)(-5.25,-1.2)to[bend left=40](-4.25,-1.2) (-5,2)to[bend right=60](-3.5,2)(-4.75,1.8)to[bend left=40](-3.75,1.8);
			\draw[Emerald,thick,rotate=-5,fill=green!10](-2,1) ellipse (.5 and 1) (-2,1.5)\ww(-2,.5)\ww;
			\draw[Emerald,thick,rotate=-70,fill=green!10](.5,.5) ellipse (.75 and 1.5) (.5,.75)\ww(.5,1.3)\ww(.5,.25)\ww(.5,-.3)\ww;
            \draw (3,-1)\ww;
			\draw[Emerald,thick,fill=green!10](5,1) circle (1) (5,1.4)\ww (4.6,0.6)\ww (5.4,.6)\ww;
            \draw[black,thick,fill=white](-3.5,-2.5)  circle (.6);
			\draw[font=\tiny]
			(-2.3,2.7)node{$\Sigma^1$}
			(-1,-1.5)node{$\Sigma^2$}
			(5,-.5)node{$\Sigma^3$};
            \draw[font=\tiny] (-4.1,-2.5)node{$\bullet$}
                (-2.9,-2.5)node{$\bullet$}
                (-3.5,-3.1)node{$\bullet$}
                (-3.5,-1.9)node{$\bullet$};
			\draw(8,0)node{$\rightsquigarrow$}(3,-4)node[font=\large]{$\subsur_\Tri\subset\surfo$};
			\begin{scope}[shift={(16,0)}]
				\draw[thick,fill=cyan!10](-4,-4)to[out=180,in=-90](-7,-1)to[out=90,in=180](-3,4)to[out=0,in=180](2,2)to[out=0,in=180](5,3)to[out=0,in=90](7,1)to[out=-90,in=0](3,-3)to[out=180,in=0](0,-2)to[out=180,in=0](-4,-4);
				\draw[thick](-5.5,-1)to[bend right=60](-4,-1)(-5.25,-1.2)to[bend left=40](-4.25,-1.2) (-5,2)to[bend right=60](-3.5,2)(-4.75,1.8)to[bend left=40](-3.75,1.8);
				\draw[red,fill=white](3,-1) \ww node[above]{$_{w_4=1}$};
				\draw[red,thick,fill=white](-2,1) circle(.15) node[above]{$_{w_1=2}$};
				\draw[red,very thick,fill=white](5,1) circle(.2) node[above]{$_{w_3=3}$};
				\draw[red,ultra thick,fill=white](0,-1) circle(.25)
				(0,-.8) node[above]{$_{w_2=4}$};
				\draw(3,-4)node[font=\large]{$\colsur$};
                \draw[black,thick,fill=white](-3.5,-2.5) circle (.6);
                \draw[font=\tiny] (-4.1,-2.5)node{$\bullet$}
                (-2.9,-2.5)node{$\bullet$}
                (-3.5,-3.1)node{$\bullet$}
                (-3.5,-1.9)node{$\bullet$};
            \end{scope}
		\end{tikzpicture}
	}
\caption{Example: collapse of a subsurface $\Sigma$ with three 
connected components of genus $0$ from a surface $\surfo$ of 
genus $2$ with $10$ decorations and four marked points on the 
boundary component. 
}\label{fig:collision}
\end{figure}
The underlying surface of $\colsur$ is the compactification of 
$\surf\setminus\ol{\Sigma}$ by points. 
The new points, together with the decorations 
in $\surf_\Tri\setminus{\Sigma}_\Tri$, 
are taken as the decorations of $\colsur$, 
denoted again $\Tri$ with abuse of notation. The marked points 
of $\colsur$ are inherited from $\surfo$ and coincide with those 
in $\mathbf{M}$. The resulting DMS 
is enriched with a map $\w\colon\Tri\to\mathbb{Z}_{\ge-1}$, 
called a \emph{weight function}, such that the 
condition \eqref{eq:aleph} is replaced by
\begin{equation}\label{eq:aleph2}
  \sum_{Z\in\Tri}\w(Z)=4g+2|\partial\surfo|+|\mathbf{M}|-4\ge1.
 \end{equation}
Under 
suitable conditions on $\surfo$ and $\Sigma_\Tri$, the weight 
function on $\colsur$ takes values in $\Z_{\geq 1}$. 
Very 
roughtly speaking, $\TT$ induces a (decorated) \lq\lq partial 
triangulation\rq\rq\ and a \lq\lq mixed-angulation\rq\rq\ 
of $\colsur$, that is a tiling of $\colsur$ into 
once-decorated polygons. A notion of forward flip for a 
mixed-angulation is available, that is used to introduce (\cite[{Section 5}]{BMQS})
an exchange graph $\EGb(\colsur)$ of mixed-angulations. 
We refer to \cite{BMQS} for more details, and in particular 
for the description 
of $\pvd(e\Gamma e)\simeq \pvd(\Gamma)/\pvd(\Gamma_I)$ as a 
triangulated category associated with $\colsur$. 
\par 
Recall the 
definition of $\EGb(\pvd(e\Gamma e))$ in Notation \ref{EG_Gamma}. 
Similarly to 
Theorem \ref{thm:QQ} we have the following 
Theorem \ref{thm_BMQS},\,1), 
depending on the choice of the initial partial 
triangulation and its refinement to a 
triangulation $\TT_0$ of $\surfo$. For part 2), 
we let $m$ be the cardinality of $I^c=Q_0\setminus I$. 
\begin{theorem}[{\cite[Theorem 6.9 and Proposition 5.13]{BMQS}}]\label{thm_BMQS} Set $\Gamma=\Gamma(Q_{\TT_0},W_{\TT_0})$ as above and $J=J(Q_{\TT_0},W_{\TT_0})$, for a fixed refining triangulation $\TT_0$, and 
let $\EGp(\pvd \Gamma)$ be the principal 
component of $\EG(\pvd \Gamma)$ containing $\modules J$. Then \begin{enumerate}
\item There is an isomorphism of graphs
\[
\EGb(\colsur) \simeq \EGb(\pvd(e\Gamma e)),\]
sending a partial triangulation $\mathbb{A}$ of $\colsur$ to 
a finite heart of a bounded t-structure of $\pvd(e\Gamma e)$ 
whose simples (up to isomorphism) correspond to internal edges 
of $\mathbb{A}$, 
and associating to any (forward) flip of edges the tilting 
of the corresponding heart at 
a simple object.
\item  The graph $\EGb(\pvd (e\Gamma e))$ 
is $(m,m)$-regular, i.e., each vertex has exactly $m$ outgoing 
and $m$ incoming directed edges, and consists of 
connected components of $\EG(\pvd e \Gamma e)$.
    \end{enumerate}
\end{theorem}

\section{Hearts of $\pvd(e\Gamma e)$}
\label{sec_tstructures}


We continue to use the terminology from the prevoius 
sections: $(Q,W)$ is a finite non-degenerate quiver with potential, 
with set of vertices $Q^0=I\sqcup I^c$, 
$\Gamma=\Gamma(Q,W)$ is its Ginzburg algebra, $e=\sum_{j\in I^c}e_j$ 
is a fixed idempotent associated with the complement of 
the subquiver $(Q_I,W_I)$. The exchange 
graphs $\EGp(\pvd\Gamma)$ and $\EGb(\pvd(e\Gamma e))$ are defined in 
Notation \ref{EG_Gamma}. Corollary \ref{any_heart_quot} tells us that 
all hearts obtained by 
titling at \emph{any} torsion pair a finite heart 
in $\EGb(\pvd(e\Gamma e))$ are of 
quotient type, i.e., they are induced by bounded 
t-structures in $\pvd\Gamma$, not necessarily finite.
\par \medskip
In this section, we assume that $(Q,W)$ is a 
quiver with potential \emph{of geometric type}, and we 
study $\EGb(\pvd(e\Gamma e))$ and \emph{simple} tilts, i.e., tilts at 
simple objects. The main result is Theorem \ref{quot_graph}, 
that shows that $\EGb(\pvd(e\Gamma e))$ can be read off 
from $\EGb(\pvd\Gamma)$. It relies in the procedure described in 
the proof of Proposition \ref{ExConvRep} for which the assumption that 
$(Q,W)$ is of geometric type is crucial. 
The content of this section is complementary to the study 
of the exchange graph $\EGb(\pvd(e\Gamma e))$ in relation 
with the exchange graph of $\colsur$ and $\surfo$, recalled in 
Theorem \ref{thm_BMQS}.

\subsection{Lifting simple tilts}\label{subsec_lifting_simple_tilts}
The goal of this subsection is to give a constructive 
explanation 
of the fact that we can ``lift'' any simple tilt 
of $\ol{\cl H}$ 
in $\EGb(\pvd (e\Gamma e))$ to a simple tilt 
in $\EGp(\pvd \Gamma)$ 
(Proposition \ref{ExConvRep}), thanks 
to the geometric description. 
We do this by 
showing that, for any simple $\ol{S}\in\ol{\calH}$, we can 
find a 
``convenient'' heart $\cl H$ in $\EGp(\pvd \Gamma)$ 
whose quotient 
is $\ol{\cl H}$. 
Being ``convenient'' with respect to $S$ for a 
heart in fact means to satisfy the hypotheses of 
Proposition \ref{prop_lift_simple_tilt} below. 
A similar 
argument appeared in \cite{BMQS}, using refinements of more 
general 
\lq\lq mixed-angulations\rq\rq\ of a wDMS. 
Proposition \ref{cor_lifting_simplefixed} can later be compared 
with Lemma \ref{thm_J}.
\par\medskip
We denote
\[\cl V:= \pvd\Gamma_I,\] 
and we fix a finite heart $\calH \in \EGp(\pvd \Gamma)$ 
compatible with $\cl V$, i.e., such that 
$$\cl K^{\cl H}:=\cl H\cap\cl V$$ is a Serre 
subcategory of $\calH$. Then the following 
proposition holds. 

\begin{prop}\label{prop_lift_simple_tilt}
Suppose that 
$S\in\calH\setminus(\cl H\cap\cl V)$ is a simple object 
such that
\begin{equation}\label{no_arrow_eq2} 
    \dim \Ext^1_{\cl H}(T,S)=0 \text{ for any } 
    T\in \cl H\cap\cl V.
\end{equation}
Then $\mu_S^\sharp\calH$ is also compatible 
with $\cl V$ and
$\mu_S^\sharp\ol{\calH}\simeq \ol{\mu_S^\sharp\calH}$.
\end{prop}
\begin{proof}
It is an application of \cite[Proposition 6.8]{BMQS}, since 
any simple object in a finite heart $\cl H$ of $\pvd\Gamma$ is 
rigid.
\end{proof}
\par
We assume now that $(Q,W)$ is of geometric type for some 
decorated marked surface $\surfo$, and that the finite heart $\cl H$ 
is associated with a triangulation $\TT$ of $\surfo$
$$\cl H=\modules J(Q_\TT,W_\TT).$$
Recall that 
its simple objects can be identified with a 
subset of the edges 
of the triangulation $\TT$. 
The combinatorial counterpart of 
equation \eqref{no_arrow_eq2} above is given in 
Lemma \ref{lem:no arrow}. Moreover, we can use the geometric model 
to prove the following two propositions.

\begin{prop}\label{ExConvRep}
	Let $S \in \calH\setminus \cl K^{\cl H}$ be a 
    simple object. Then there exists a finite sequence of 
    simple tilts at objects in $\cl V=\pvd\Gamma_I$ such that the 
    resulting heart is $\mu^\sharp_\torsionfree\calH$, for 
    some torsion-free class 
    $\torsionfree\subset \cl K^{\cl H}$, and 
    satisfies 
    \begin{itemize}
        \item[i)] $S\in\mu^\sharp_\torsionfree\calH$ and it 
        is simple,
        \item[ii)]$\dim \Ext^1(T,S)=0$ for any 
        $T\in \mu^\sharp_\torsionfree\calH \cap \cl V$, 
        and 
        \item[iii)] the quotient abelian categories 
        $\mu^\sharp_\torsionfree\calH/
        \left(\mu^\sharp_\torsionfree\calH \cap \cl V\right)$ 
        and $\calH/\cl K^{\cl H}$ coincide:
        \[\ol{\mu^\sharp_\torsionfree\calH}=\ol{\calH}.\]
    \end{itemize}
\end{prop}
\par
\begin{proof}
We consider the triangulation~$\TT$ of $\surfo$ induced 
by~$\calH$. We call $S$ the edge corresponding to the simple 
object $S$ in $\calH\setminus\cl K^{\cl H}$, and 
color in
blue (resp. black) the edges corresponding to the simples 
of $\calH$ in 
$\calH\cap\cl V$ 
(resp. $\calH\setminus\cl K^{\cl H}$). Black edges 
define a 
tiling of $\surfo$ into polygons, and blue edges refine this 
tiling to the triangulation $\TT$. 
Without loss of generality, we may restrict our attention to 
two
adjacent polygons $P_k$ and $P_\ell$, with $k$ and $\ell$ 
vertices 
respectively, sharing $S$ as common edge. We denote by $d_k$ 
(resp.~$d_\ell$) the diagonal of $P_k$ (resp.~$P_\ell$) which, 
together with~$S$ and its counter-clockwise neighbouring 
edge of the 
polygon, form a triangle. 
Examples of such a configuration are displayed in 
Figure~\ref{cap:klgons}. 
The common edge $S$ is displayed in red, and the 
diagonals $d_k,d_\ell$ are displayed as black dotted lines.
\begin{figure}[ht]
	\begin{tikzpicture}[scale=0.8,cap=round,>=latex]		
\coordinate (a) at (45+22.5:2cm);
		\begin{scope}[shift={(3.1734cm, 0 cm)}]
		
		\draw[thick] (30:1.5307cm) -- (90:1.5307cm);
		\draw[thick] (90:1.5307cm) -- (150:1.5307cm);
		\draw[thick] (210:1.5307cm) -- (270:1.5307cm);
		\draw[thick] (270:1.5307cm) -- (330:1.5307cm);
		\draw[thick] (330:1.5307cm) -- (390:1.5307cm);
        \coordinate (b) at (270:1.5307cm);
		
		\foreach \x in {30,90,...,330} {
			\filldraw[black] (\x:1.5307cm) circle(0.5pt);
		}	
        \draw[thick, red, dotted] (a)--(b)  node [midway, fill=white] {};
        \draw[thick, red] (150:1.5307cm) -- (210:1.5307cm) node [midway, fill=white] {$S$};
		\draw[dotted, thick] (270:1.5cm) -- (150:1.5307cm);
		\draw[blue] (210:1.5307cm) -- (30:1.5307cm);
		\draw[blue] (90:1.5307cm) -- (210:1.5307cm);
        \draw[blue] (30:1.5307cm) -- (270:1.5307cm);

		\end{scope}

        \draw[thick] (0+22.5:2cm) -- (45+22.5:2cm);
		\draw[thick] (45+22.5:2cm) -- (90+22.5:2cm);
		\draw[thick] (90+22.5:2cm) -- (135+22.5:2cm);
		\draw[thick] (135+22.5:2cm) -- (180+22.5:2cm);
		\draw[thick] (180+22.5:2cm) -- (225+22.5:2cm);
		\draw[thick] (225+22.5:2cm) -- (270+22.5:2cm);
		\draw[thick] (270+22.5:2cm) -- (315+22.5:2cm);
		
		\foreach \x in {22.5,67.5,...,347.5} {
			\filldraw[black] (\x:2cm) circle(0.5pt);
		}
	
		\draw[thick, dotted] (315+22.5:2cm) -- (45+22.5:2cm);
		\draw[blue] (0+22.5:2cm) -- (135+22.5:2cm);
		\draw[blue] (45+22.5:2cm) -- (135+22.5:2cm);
		\draw[blue] (0+22.5:2cm) -- (180+22.5:2cm);
		\draw[blue] (0+22.5:2cm) -- (225+22.5:2cm);
		\draw[blue] (225+22.5:2cm) -- (315+22.5:2cm);
    
    \begin{scope}[shift={(8 cm, 0 cm)}]
        \draw[thick] (45:1.6cm) -- (135:1.6cm);
        \draw[thick, red] (135:1.6cm) -- (225:1.6cm) node [midway, fill=white] {$S$};
        \draw[thick] (225:1.6cm)--(315:1.6cm);
        \draw[thick, red] (315:1.6cm)--(405:1.6cm)  node [midway, fill=white] {$S$};
        \draw[blue] (135+90:1.6cm)--(315+90:1.6cm);
        \draw[thick,dotted] (45+90:1.6cm)-- (225+90:1.6cm);
        \draw[thick, dotted, red] (357+90:1.1cm)--(225+90:1.6cm);
        \draw[thick, dotted, red] (177+90:1.1cm)--(45+90:1.6cm);
    \end{scope}
	\end{tikzpicture}
	\caption{On the left, an example of $P_k\cup P_\ell$ 
    with $k=8$ and $\ell=6$. On the right, a polygon with 
    two edges identified.} \label{cap:klgons}
\end{figure}
\par
In order to get a torsion-free class $\torsionfree$ as 
required in the 
lemma, we construct a suitable triangulation 
of $P_k\cup P_\ell$ by 
iteratively flipping some blue edges 
inside $P_k\cup P_\ell$. 
Using the correspondence between flips and simple tilts, this 
guaranties that the resulting quotient category does not 
change, 
i.e., condition iii). Since flips will occur at most once 
at each 
edge, the procedure will define a torsion-free 
class $\torsionfree \subset \cl K^{\cl H}$. 
\par We note 
that the heart corresponding to a triangulation of $\surfo$ 
satisfies condition ii) if and only if it contains $d_k$ 
and $d_\ell$ (this is the no-arrow condition of 
Lemma \ref{lem:no arrow} and the construction of the 
quiver from the triangulation $\TT$). 
\par
Consider the finitely many edges $e_i$ that intersect, out 
of the 
extremes, the diagonals $d_k$ and $d_\ell$, and label them 
counter-clockwise. 
We mutate at these edges in lexicographic order. The 
diagonal $d_k$ 
subdivides the $k$-gon into a triangle and a $(k-1)$-gon 
sharing 
$d_k$ as an edge and its extremes as vertices. Any $e_i$ 
intersecting $d_k$ has the characterizing vertex of the 
triangle as 
an endpoint. Flipping $e_i$ gives us a new edge, with two 
different 
endpoints, that lies completely in the $(k-1)$-gon. 
Once all crossing edges have been flipped exactly once, 
the 
resulting triangulation has $d_k$ among its edges. The same 
procedure will be applied to the other polygon $P_\ell$. Note 
that 
the procedure works also when two black edges are identified, 
or, 
e.g., in the example displayed on the right of 
Figure \ref{cap:klgons}, where the black dashed edge is the 
resulting edge after flipping the blue edge.
\par
Let $S_i$ be the simple object of $\calH$ corresponding 
to $e_i$. 
Saying $\calH_0=\calH$ and 
$\calH_i=\mu_{S_i}^\sharp\calH_{i-1}$ the 
result of a single simple tilt, we note that at each 
step, $S$ stays 
simple in $\calH_i$, hence in $\mu^\sharp_\torsionfree\calH$, 
which is 
condition ii). 
\end{proof}

\begin{prop}\label{cor_lifting_simplefixed} 
    Let $\cl H\in \EGp(\pvd\Gamma)$ and $S \in \calH\setminus
    \cl K^{\cl H}$ 
    as in Proposition \ref{prop_lift_simple_tilt}. 
    Then, there exists $\cl G\in \EGp(\pvd\Gamma)$ compatible 
    with $\pvd\Gamma_I$ such that 
    \begin{itemize}
    \item[i)] $\cl G\supset \cl K^{\cl H}$, and
    \item[ii)]  $\ol{\cl G}= \mu^\sharp_{\ol{S}}\ol{\cl H}$.
    \end{itemize} 
\end{prop}
\begin{proof}
Let $\torsionfree \subset \cl K^{\cl H}$ be 
the torsion free class constructed in the proof of 
Proposition \ref{ExConvRep}
and $\torsion$ its corresponding torsion class in $\cl H$. 
The heart $\cl H$ is compatible with $\pvd\Gamma_I$ by 
hypotesis, while the hearts 
\[\cl G':= \mu_\torsionfree^\sharp\cl H=
\torsionfree[1]*\torsion, \text{ and}\quad\cl G'' := \mu_S^\sharp\cl G'\]
are compatible with $\pvd\Gamma_I$, by Lemma \ref{quot_tilted_hearts} 
because $\torsionfree\subset 
\pvd\Gamma_I$, and by Proposition \ref{prop_lift_simple_tilt}, respectively.
This means that the corresponding bounded 
t-structures in $\pvd\Gamma$ induce bounded t-structures 
in $\pvd\Gamma(e\Gamma e)$, and, in 
particular, that $\cl K':=\cl G'\cap\pvd\Gamma_I\subset \cl G'$ and 
$\cl K'':=\pvd\Gamma_I \cap \cl G''\subset \cl G''$ are 
Serre subcategories of the hearts.
We 
can therefore consider the restricted bounded t-structures in 
$\pvd\Gamma_I$ as well as the restricted torsion pairs, and 
trace them. We have that
\begin{itemize}
\item the torsion pair $(\torsionfree,\torsion)$ in $\cl H$ induces the 
torsion pair $(\torsionfree, \torsion\cap \cl K^{\cl H})$ 
in $\cl K$, with 
\[\mu_\torsionfree^\sharp\cl K^{\cl H} = 
\torsionfree[1]*(\torsion\cap\cl K^{\cl H}) = \cl K'.\]
\end{itemize}
Since $S\in \cl G'$ is simple and not contained in $\cl K'$, 
then $\cl K'\subset {^\perp S}$. Therefore\begin{itemize}
\item the torsion pair $(\operatorname{filt}(S), {^\perp S})$ 
in $\cl H'$ induces the torsion pair $(\cl K',0)$ in $\cl K'$, 
so that 
\[\cl K'' = \mu_S^\sharp \cl K' 
= \cl K'.\]
\end{itemize}
Last, the torsion pair 
$(\torsionfree[1], {\torsionfree[1]}^{\perp_{\cl G''}})$ restricts to 
$(\torsionfree[1], \torsion\cap\cl K^{\cl H})$ 
in $\cl K''\subset\cl G''$. 
We can therefore perform a \emph{backward} tilt 
at $\torsionfree[1]$ of $\cl G''$ to 
get 
\[\cl G:=\mu^\flat_{\torsionfree[1]}\cl G'',\] 
compatible with $\pvd\Gamma_I$, and containing 
$\mu^\flat_{\torsionfree[1]}\cl K'' = 
\mu^\flat_{\torsionfree[1]}\mu^\sharp_\torsionfree\cl K^{\cl H}
=\cl K^{\cl H}$.
\end{proof}

\begin{rmk}
The dual and the corresponding analogues of 
Propositions \ref{prop_lift_simple_tilt}, \ref{ExConvRep}, 
\ref{cor_lifting_simplefixed} hold for backward tilting.
\end{rmk}

\subsection{The exchange graph $\EGb(\pvd(e\Gamma e))$}
We first recall the general definition of quotient of a graph, 
and then show that the graph $\EGb(\pvd(e\Gamma e))$ is isomorphic 
to a quotient of $\EGp(\pvd\Gamma, \pvd\Gamma_I)$.
\begin{definition}
    Let $\operatorname{G}=(V,E)$ be a (oriented) graph with 
    set of vertices $V$ and (directed) edges $E$. Let $E'\subset E$. We say that $E'$ induces the partition $P=\{P_1,\dots, P_t\}$ of $V$, where, for any pair $i\neq j$, the vertices $V_i,V_j$ are in the same subset $P_r$ if and only if there is a path of (directed) edges in $E'$ connecting them. 
    
    The 
    \emph{quotient of $\operatorname{G}$ with respect to $E'$} is 
    defined as the (oriented) graph $G/E'$ whose set of vertices 
    coincides with $P$, the partition of $V$ induced by $E'$, 
    and such that there is a (directed) edge 
    connecting $P_r$ and (to) $P_s$  
    for any (directed) edge in $E \setminus E'$ connecting a 
    vertex of $V$ in $P_r$ and one in $P_s$.
\end{definition}

\begin{theorem}\label{quot_graph} 
The graph $\EGb(\pvd (e\Gamma e))$ is isomorphic to the 
quotient 
\[\EGb(\pvd (e\Gamma e))\simeq \EGp(\pvd \Gamma,\pvd \Gamma_I)/E_I,\]
of $\EGp(\pvd \Gamma, \pvd \Gamma_I)$ with respect to the 
set $E_I$ of directed edges labelled by an object in $\pvd \Gamma_I$. 
\\
It consists of connected components of $\EG(\pvd(e\Gamma e))$.
\end{theorem}
\begin{proof} 
We use the ``lifting simple tilts'' procedure of the previous 
subsection. 
First we note that there is a bijection between vertices 
of $\EGb(\pvd (e\Gamma e))$ and vertices 
of 
\[\EGb(\pvd \Gamma,\pvd \Gamma_I)/E_I.\] Indeed, any simple tilt 
at an object in $\calV:=\pvd \Gamma_I$ does not change the induced 
abelian heart in $\pvd(e\Gamma e)$.  
Propositions \ref{ExConvRep} together with \ref{prop_lift_simple_tilt} 
above show that any simple tilt in $\EGb(\pvd (e\Gamma e))$ lifts 
to a simple tilt in $\EGp(\pvd \Gamma)$ and 
in $\EGp(\pvd\Gamma,\pvd\Gamma_I)/E_I$. On the other hand, 
having an edge in $\EGp(\pvd \Gamma,\pvd \Gamma_I)/E_I$ 
connecting $P_i$ and $P_j$ means to have a simple 
tilt $\mu^\sharp_S$ relating two $\calV$-compatible 
hearts $\cl H_1$ and $\cl H_2$
in $\pvd(\Gamma)$ for some $S\in\cl H_1$ not in $\calV$. 
This is possible if and only 
if $S$ satisfies $\dim\Ext^1_{\cl H_1}(T,S)=0$ for any $T$. 
Indeed, 
if $\Sim(\cl H_1)=\{[S],[S_2],\dots, [S_n]\}$ 
and $\Sim(\cl H_1\cap\calV)=\{[S_{k}],\dots, [S_n]\}$, 
the new simples in $\cl H_2$ are $S[1]$, $S_i$ for 
any $i$ such that $\dim\Ext^1_{\cl H_1}(S_i,S)=0$, 
and $F_i=\operatorname{Cone}\left(S_i\to S\otimes\Ext^1(S_i,S)^{\vee}[1]\right)[-1]$ 
for any other $i$. Suppose that $S_k$ is replaced by $F_k$. 
Then $S_k$ in $\cl H_2$ is an extension of $S[1]$ and $F_k$ 
contradicting the hypothesis that $\cl H_2\cap\calV$ is 
Serre in $\cl H_2$.\\
The second part of the statement was proven 
in \cite[{Theorem 6.9}]{BMQS} 
and also follows from Proposition \ref{any_heart_quot}.
\end{proof}

\begin{example}[Figure \ref{EG_A3}]
As an example of Theorem \ref{quot_graph}, we consider the linear quiver 
$Q=A_3=\bullet \to \bullet \to \bullet$ with straight 
orientation and trivial potential $W=0$. Let $\calD=\pvd(Q,W)$ and 
call $X,Y,Z$ the simple dg modules generating the standard 
heart $\calH=\rep A_3$. Figure \ref{EG_A3} 
displays in gray part of the exchange graph of $\calD$, which is taken 
from \cite{kq}, and in blue and red part of the exchange graph 
of $\calD/\calV$, as a quotient graph, for two different choices of 
thick subcategories $\calV\subset \calD$. 
We explain here the notation. We call $X,Y,Z$ the simples of $\calH$ and $U,V,W$ the new simples obtained by forward tilting $\calH$ at $\bra X\ket$, $\bra Y\ket$, $\bra X,U\ket$ respectively. For any object $M\in\calD$, denote by 
$M_i=M[i-1]$, so that $M_1=M$ and $M_2=M[1]$. $\ol M$ denotes the image of $M$ under the appropriate quotient functor. 
Vertices of the graphs are labeled with the simple generators of the corresponding 
bounded heart and edges are labeled by the simple object the tilting is performed at. 
Coloured dotted lines emanate from those hearts that are compatible with 
the subcategory $\calV$, and connect them with their abelian quotient.

To construct (part of) the exchange graphs of the corresponding weighted decorated marked surfaces, that are discs with three decorations of weights $\w=(1,1,2)$, see \cite[Sec.~3.3]{BMQS}. Mixed-angulations for the two cases considered here will not look the same, despite giving rise to similar pentagon-shaped partial exchange graphs. Compare it also with \cite[Fig.~3]{Kr}.
\begin{figure}\centering
\makebox[\textwidth][c]{
\begin{tikzpicture}[scale=1.5, rotate=-117, xscale=-1,
 arrow/.style={->,>=stealth,thick},
 midlabel/.style={midway,fill=white}]
\path[gray] (0,0) node (x1)  {$_{X_1 Y_1 Z_2}$};
\path[gray] (2,1) node (x2)  {$_{X_1 Y_1 Z_1}$};
\path[gray] (-2,1) node (x3) {$_{X_2 U_1 Z_2}$};
\path[gray] (0,2) node (x4)  {$_{X_2 U_1 Z_1}$};
\path[gray] (0,4) node (x5)  {$_{W_1 Y_1 U_2}$};
\path[gray] (4,5) node (x6)  {$_{X_1 Y_2 V_1}$};
\path[gray] (2,5) node (x7)  {$_{X_2 Y_2 W_1}$};
\path[gray] (-2,5) node (x8) {$_{W_2 Y_1 Z_1}$};
\path[gray] (-4,5) node (x9) {$_{U_2 Y_1 Z_2}$};
\path[gray] (0,6) node  (x10)  {$_{V_1 Y_2 W_2}$};
\path[gray] (0,8) node  (x11) {$_{X_2 V_2 Z_1}$};
\path[gray] (2,9) node  (x12) {$_{X_1 V_2 Z_1}$};
\path[gray] (-2,9) node (x13) {$_{X_2 Y_2 Z_2}$};
\path[gray] (0,10) node  (x14){$_{X_1 Y_2 Z_2}$};

\path[red] (3.6,2) node (r1) {$_{\overline{X}_1 \overline{Z}_1}$};
\path[red] (1.2,4) node (r2) {$_{\overline{X}_2 \overline{W}_1}$};
\path[red] (-1.2,6) node (r3) {$_{\overline{Z}_1 \overline{W}_2}$};
\path[red] (-3.6,8) node (r4) {$_{\overline{X}_2 \overline{Z}_2}$};
\path[red] (1.1,6.3) node (r5) {$_{\overline{X}_1 \overline{Z}_2}$};

\path[blue] (1.1,2) node (b1) {$_{\overline{Y}_1 \overline{Z}_1}$};
\path[blue] (-0.6,1) node (b2) {$_{\overline{Y}_1 \overline{Z}_2}$};
\path[blue] (2.8,5.5) node (b3) {$_{\overline{Y}_2 \overline{V}_1}$};
\path[blue] (1.6,8) node (b4) {$_{\overline{V}_2 \overline{Z}_1}$};
\path[blue] (-0.6,9.2) node (b5) {$_{\overline{Y}_2 \overline{Z}_2}$};

\draw[arrow,gray]
(x2) edge node[midlabel] {\tiny{$Z_1$}} (x1)
     edge node[midlabel] {\tiny{$X_1$}} (x4)
     edge node[midlabel] {\tiny{$Y_1$}} (x6)
(x1) edge node[midlabel] {\tiny{$X_1$}} (x3)
     edge [bend right=-7, dash pattern=on 6pt off 2pt, thin] node[midlabel] {\tiny{$Y_1$}} (x14)
(x4) edge node[midlabel] {\tiny{$U_1$}} (x5)
     edge node[midlabel] {\tiny{$Z_1$}} (x3)
(x3) edge node[midlabel] {\tiny{$U_1$}} (x9)
(x5) edge node[midlabel] {\tiny{$Y_1$}} (x7)
     edge node[midlabel] {\tiny{$W_1$}} (x8)
(x6) edge node[midlabel] {\tiny{$X_1$}} (x7)
     edge node[midlabel] {\tiny{$V_1$}} (x12)
(x7) edge node[midlabel] {\tiny{$W_1$}} (x10)
(x8) edge node[midlabel] {\tiny{$Y_1$}} (x10)
     edge node[midlabel] {\tiny{$Z_1$}} (x9)
(x9) edge node[midlabel] {\tiny{$Y_1$}} (x13)
(x10) edge node[midlabel] {\tiny{$V_1$}} (x11);
\draw[arrow,gray]
(x11) edge node[midlabel] {\tiny{$Z_1$}} (x13)
(x14) edge node[midlabel] {\tiny{$X_1$}} (x13)
(x12) edge node[midlabel] {\tiny{$X_1$}} (x11)
      edge node[midlabel] {\tiny{$Z_1$}} (x14);
\draw[arrow,red]
    (r1) edge node[midlabel] {\tiny{$\overline{X}_1$}} (r2)
    (r2) edge node[midlabel] {\tiny{$\overline{W}_1$}} (r3)
    (r3) edge node[midlabel] {\tiny{$\overline{Z}_1$}} (r4)
    (r1) edge[bend right=0] node[midlabel] {\tiny{$\overline{Z}_1$}} (r5)
    (r5) edge node[midlabel] {\tiny{$\overline{X}_1$}} (r4)
    ;
\draw[red,dash pattern=on 2pt off 2pt] (r1)
    edge (x2) edge (x6);
\draw[red,dash pattern=on 2pt off 2pt] (r2)
    edge (x5) edge (x7);
\draw[red,dash pattern=on 2pt off 2pt] (r3)
    edge (x8) edge (x10);
\draw[red,dash pattern=on 2pt off 2pt] (r4)
    edge (x9) edge (x13);
    \draw[red,bend left=0,dash pattern=on 2pt off 2pt] (r5)
    edge (x1);
    \draw[red,dash pattern=on 2pt off 2pt] (r5)
    edge (x14);
\draw[arrow,blue] (b1)
    edge node[midlabel] {\tiny{$\overline{Z}_1$}} (b2)
    edge node[midlabel] {\tiny{$\overline{Y}_1$}} (b3);
\draw[arrow,blue] (b3)
    edge node[midlabel] {\tiny{$\overline{V}_1$}} (b4);
\draw[arrow,blue] (b4)
    edge node[midlabel] {\tiny{$\overline{Z}_1$}} (b5);
\draw[arrow,blue] (b2)
    edge[bend left=5] node[midlabel] {\tiny{$\overline{Y}_1$}} (b5);
\draw[blue,dash pattern=on 2pt off 2pt] (b1)
    edge (x2) edge (x4);
\draw[blue,dash pattern=on 2pt off 2pt] (b2)
    edge (x1) edge (x3);
\draw[blue,dash pattern=on 2pt off 2pt] (b3)
    edge (x6) edge (x7);
\draw[blue,dash pattern=on 2pt off 2pt] (b4)
    edge (x11) edge (x12);
\draw[blue,dash pattern=on 2pt off 2pt] (b5)
    edge (x13) edge (x14);
\end{tikzpicture}
}    
\caption{(Part of) the exchange graphs $\EGb(\calD)$ and $\EGb(\calD/\cl V)$ for $\calD=\pvd(A_3)=\bra X_1,Y_1,Z_1\ket$ and $\cl V=\bra X_1\ket$ (blue) or $\cl V=\bra Y_1\ket$ (red). The dotted lines stand for taking the quotient of hearts $\calA/(\calA\cap\cl V)$ when the t-structure $\calA$ is compatible with $\cl V$. 
}
\label{EG_A3}
\end{figure}
\end{example}

The graph $\EG(\pvd(e\Gamma e))$ is 
not necessarily connected. An example of disconnected graph 
is provided in \cite[Section 5.5]{BMQS} in terms of graphs of an 
associated weighted decorated marked surface.

\section{Simple-projective duality}
\label{sec_duality}

In this section, we relate simple tilts in $\pvd (e\Gamma e)$ 
with \emph{silting mutations} (Definition \ref{def_silting} below) 
in the perfect derived 
category $\per (e\Gamma e)$, and give an isomorphism between the 
exchange graph
$\EGb(\pvd(e\Gamma e))$ and the silting exchange graph of $\per(e \Gamma e)$,
see Theorem~\ref{thm:SPdual}. For, 
we first construct a bijection between sets of silting 
objects and 
finite hearts of $\pvd(e\Gamma e)$ (Proposition~\ref{le:dual}), 
and then we show that the bijection is compatible with 
silting and tilting mutations when $\Gamma=\Gamma(Q,W)$ for 
a quiver of geometric type. In fact, we use 
simple-projective duality for $\pvd\Gamma$ and $\per\Gamma$ and 
show that it descends to $e\Gamma e$.
\par 
We start by recalling some definitions and 
simple-projective duality for $\Gamma$.

\subsection{Simple-projective duality for $\pvd\Gamma$ and $\per\Gamma$} \label{sec:simpleprojD}

 Let $\C$ be an 
additive category with a full
subcategory $\calX$.
A \emph{left $\calX$-approximation} of an object $C$ in $\C$ 
is a morphism $f:C\to X$
 with $X\in\calX$, such that $\Hom(f,X')$ is an epimorphism 
 for any $X'\in\calX$.
A \emph{minimal left $\calX$-approximation} is a left 
approximation $f$ that is moreover
left-minimal, i.e., for any $g:X\to X$ such 
that $g\circ f = f$, the morphism $g$ is an isomorphism.

\begin{definition}[{\cite{KV,AI}}]\label{def_silting} Let $\C$ 
	be a triangulated category.\begin{enumerate}
\item A full subcategory $\calY$ of $\C$ is called a 
\emph{partial silting subcategory} of $\C$
	if $$\Hom^{>0}(\calY,\calY)=0.$$ It is said \emph{silting} 
	if, furthermore, $\thick(\calY)=\cl C$.\footnote{These 
	are often 
	called \emph{pre-}silting subcategories.}
\item A \emph{(partial) silting object}
	 $\bY$ is a direct sum of
	non-isomorphic indecomposable objects $Y_i$ of $\C$ such 
	that $\add \bY$ is
	a (partial) silting subcategory\footnote{Note that in 
	the literature this definition usually
	corresponds to a \emph{basic} (partial/pre) silting object.}.
\item The \emph{forward mutation} $\fmu_{Y_k}$ of a partial 
silting object $\bY=\bigoplus_i Y_i$ with respect to a summand $Y_k$ is an operation that produces another partial silting object
	$\fmu_{Y_k}\bY:=Y_k^\sharp\oplus\bigoplus_{i\neq k} Y_i$, 
	defined by 
	\begin{equation}\label{eq:fmu}
		Y_k^\sharp := \Cone \Big(Y_k \xrightarrow[]{f} 
		\bigoplus_{i\neq k} \Irr(Y_k,Y_i)^{\vee}\otimes Y_i \Big),
	\end{equation}
	where $\Irr(X,Y)$ is the space of irreducible 
	maps $X\to Y$ in $\add(\bY)$, and $f$ is the left 
	minimal $\add(\bigoplus_{i\neq k} Y_i)$-approximation 
	of $Y_k$.
\end{enumerate}
\end{definition}

If the Grothendieck group of $\C$ is finite (e.g., when $\C=\per\Gamma$),
any partial silting object has finitely many summands.
\par
\begin{definition}\label{def_SEG}	For the perfect derived category $\C=\per \Lambda$ of a non-positive dg
	algebra $\Lambda$, we denote by $\SEG\C$ the \emph{silting exchange graph}, i.e.,  the graph whose vertices are silting objects
	and whose directed edges are forward mutations between them, and by $\SEGp$ the connected component
	containing $\Lambda$
	as a vertex
	(same notation as for
	the exchange graphs $\EG$s).
\end{definition}
\par

The following theorem, which we refer to as \lq\lq simple-projective duality\rq\rq ,
is due to Keller-Nicol\`{a}s in our setting. The original proof has never been
published, but it is very similar to the proofs of analogous results in slightly
different contexts that can be found for instance in~\cite{KV, kq, KoY}.
It holds for the non-positive dg Ginzburg algebras $\Gamma(Q,W)$ of finite quivers with
potential. As a bijection of sets, i.e., irrespectively of 
compatibility with mutations, it is proved in the more general 
setting of locally finite non-positive dg algebras 
by \cite[Theorem 3.5]{fushimi}, which also covers the case 
of $e\Gamma e$.
\par
Recall that $\pvd\Gamma\subset \per \Gamma$.
\par
\begin{theorem}[Simple-projective duality for $\pvd(\Gamma)$ and $\per(\Gamma)$]\label{thm:KN2}
	There is an isomorphism between oriented graphs
	\begin{equation}\label{eq:s-p-dual}
		\dual_\Gamma\colon \EGp(\pvd(\Gamma))\to\SEGp(\per(\Gamma)),
	\end{equation}
	sending a finite heart $\calH$ with iso-classes of simples $\Sim\calH=\{[S_1],\dots, [S_n]\}$ to a silting object $\bY_\h=\bigoplus_{i=1}^n Y_i$, satisfying
	\begin{gather}
		\Hom_{\per\Gamma}(Y_i,S_j)=\delta_{ij}\acf\label{eq:Hom01}
	\end{gather}
 and
 \begin{gather}
		\Irr_{\add(\bY)}(Y_j,Y_i)\simeq \Ext^1(S_i,S_j)^{\vee}.\label{eq:Irr=Ext}
	\end{gather}
	In particular, the silting mutation corresponds to simple tilting.
\end{theorem}

The key idea of the proof of Theorem~\ref{thm:KN2} is that the category $\per\Gamma$
	can be realized as $\per\Gamma_{\h}$, where $\Gamma_{\h}$ is the dg endomorphism
	algebra of the silting object $\bY_\h$ for any $\h\in \EGp(\pvd(\Gamma))$. Then the
	summands $Y_i$ of $\bY_\h$ are
	projective $\Gamma_{\h}$-modules and the corresponding simples~$S_i$ in $\h$ are
	simple $\Gamma_{\h}$-modules, for which \eqref{eq:Hom01} and \eqref{eq:Irr=Ext}
	hold.

\subsection{Simple-projective duality for $\pvd(e \Gamma e)$ and $\per(e\Gamma e)$}

We show that Theorem \ref{thm:KN2} is preserved under 
localization. Proposition \ref{le:dual} below states a correspondence between 
finite hearts of bounded t-structures and silting objects. 
Theorem \ref{thm:SPdual} promotes this bijection to an isomomorphism 
of graphs, in the case that $\Gamma=\Gamma(Q,W)$ comes from a quiver 
with potential with geometric description. 
\begin{rmk*}
Let $\calK:=\modules J_I =\<S_i \mid i\in Q_I\>$ be the Serre
subcategory of $\calA:=\modules J$ associated 
with $I\subset Q_0$, and  $\calV:=\pvd(\Gamma_I)=
\thick \cl K\subset \pvd\Gamma$. 
\end{rmk*}

\begin{rmk*}
If $\operatorname{G}$ is a graph, $|\operatorname{G}|$ 
denotes its set of vertices.
\end{rmk*}

Recall that the principal part $\EGb(\pvd(\pvc))$ 
of $\EG(\pvd(\pvc))$
consists of the quotients of $\calV$-compatible hearts 
in $\EGp(\pvd(\Gamma))$, the principal part of the exchange graph 
of $\pvd\Gamma$ containing $\cl A=\modules J$. 
\par\medskip
Before proving the main results, we rephrase a 
result by H.\,Jin, originally expressed in terms of \emph{simple 
minded collections}, corresponding to collections of simple objects 
(up to isomorphism) in finite hearts of bounded t-structures. 
\begin{lemma}\label{thm_J} Each finite heart 
$\ol{\cl H}\in \EGb(\pvd(e\Gamma e))$ 
can be lifted uniquely to a finite heart $\cl H\in \EGp(\pvd \Gamma)$ 
such that $\Sim\cl H$ contains $\Sim\calK$.
\end{lemma}
\begin{proof}
It is an application of \cite[Theorem 3.1]{J}, using the 
correspondence 
between finite hearts and simple minded collections 
in $\pvd\Gamma$ (see e.g.\ \cite{KoY}).
\end{proof}
We denote by 
$|\EGb_\calK(\pvd(\Gamma))|$ the sub-set of all such lifts:
\begin{equation}\label{Jin_map}
	\xymatrix{|\EGb_\calK\pvd(\Gamma)| \ar@{^{(}->}[r] & |\EG^\circ\pvd\Gamma|\\
	|\EGb\pvd(e\Gamma e)| \ar[u]^{1:1}_{\text{\cite{J}}}  
	}
\end{equation}
\begin{rmk} If $\Gamma = \Gamma(Q, W )$ for some quiver 
	with potential with geometric description, 
then Section \ref{subsec_lifting_simple_tilts} 
provides a constructive way of finding the lift $\cl H$ 
of $\ol{\cl H}$ containing $\cl K$. 
The procedure possibly 
depends on the choice of a reference heart $\cl A\supset\cl K$ 
for any connected 
component of $\EGp(\pvd \Gamma, \pvd \Gamma_I)$.
\end{rmk}
\end{comment}

We also need some preliminary results and to set some notation. 
Recall 
that $\per (e\Gamma e)$ is orthogonal to $\calV=\pvd\Gamma_I$ 
inside $\per\Gamma$, by Lemma \ref{lem:perpV}. 
\begin{definition}\label{def_pSEG}
A partial silting object $\bX$ in $\per(\Gamma)$ is 
called \emph{$\calV$-perpendi\-cular}
if 
\[\Hom_{\per\Gamma}(\bX,\calV)=0 \text{ and }\thick(\bX)=
{^{\perp_{\per(\Gamma)}}}\calV.\]
We define 
\[\pSEGV \left(\per\Gamma\right)
\]
as the graph whose vertices are $\calV$-perpendicular 
\emph{partial} silting objects $\bX$ and whose directed edges  
are forward mutations. We call it the 
\emph{$\calV$-perpendicular silting exchange graph}.
\end{definition}

\begin{lemma}\label{lem:obv}
There is a canonical bijection
\[
    \SEG(\per(e\Gamma e))\, \simeq \, \pSEGV (\per\Gamma)\,.
\]
\end{lemma}
\begin{proof}
Any silting object $\bY$ in $\per(\pvc)$ is naturally a 
partial silting object in $\per\Gamma$, because $\ell$ 
is fully faithful.
Moreover, $\bY\in \per(\pvc)$ is $\calV$-perpendicular 
because $\thick\bY=\per(\pvc)={^{\perp_{\per\Gamma}}}\calV$ 
(Lemma \ref{lem:perpV}). Hence we have an inclusion
$$|\SEG(\per(e\Gamma e))|\hookrightarrow |\pSEGV(\per\Gamma)|.$$
It is straightforward to see that this map is a bijection.\\
Since the definition of partial silting mutation  of a (partial) 
silting object $\bY$ only depends on the summands of $\bY$ and 
maps between them, and since $\ell$ is fully faithful, it is 
in fact an isomorphism of graphs.
\end{proof}

\par
\begin{definition}\label{def_seg_bullet} Fix the connected componet $\SEGp\left(\per (\Gamma)\right)$ containing $\Gamma$.
	\begin{itemize}
\item We let $\pSEGVb (\per(\Gamma))$ be the principal part 
of $\pSEGV (\per(\Gamma))$
consisting of partial silting objects that can be completed 
into
silting objects in $\SEGp(\per(\Gamma))$. 
\item We denote 
by $\SEGb(\per(\pvc))$
the full subgraph of $\SEG(\per (e\Gamma e))$ 
which is the preimage of $\pSEGVb (\per(\Gamma))$ under the 
bijection in Lemma \ref{lem:obv}, so that 
we have a resulting bijection
\begin{equation}\label{eq:id}
	\operatorname{i}\colon \SEGb (\per(\pvc)) \simeq 
	\pSEGVb (\per(\Gamma)).
\end{equation}
\end{itemize}
\end{definition}
\noindent The graph $\SEGb(\per (e\Gamma e))$ is the relevant graph appearing 
in simple-projective duality for $e\Gamma e$.
\begin{prop}\label{le:dual}
There is a bijection
\begin{gather}\label{eq:s-p-dual2}
	\iota_s\colon|\EGb(\pvd(\pvc))|\simeq|\SEGb(\per(\pvc))|.
\end{gather}
\end{prop}
\begin{proof} In order to prove the bijection, we construct explicitly a map 
\[
\ivc:|\EGb(\pvd(e \Gamma e))| \to |\pSEGVb(\per\Gamma)|
\] and compose it with the bijection \eqref{eq:id}. 
The construction of $\ivc$ involves three steps: a heart 
$\overline{\cl H}\in |\EGb(\pvd (e\Gamma e))|$ is first 
lifted to a heart $\cl H$ of $\pvd \Gamma$ containing $\calK$, 
then the silting-projective duality for $\Gamma$ 
is applied to get a silting object in $\per \Gamma$, 
which is restricted to a $\calV$-perpendicular 
silting object (diagram \eqref{ivc_diagram} below). 
To conclude the proof, the map is shown to be 
injective (\emph{step 4}), and finally bijective (\emph{step 5}).
\par
\emph{Steps 1 and 2.}  Let $\ol{\cl H}$ be a finite heart in
$\EGb(\pvd(e\Gamma e))$. By Lemma \ref{thm_J} (\cite[Theorem 3.1]{J}), 
it can be lifted uniquely to a finite heart $\cl H\in \EGp(\pvd \Gamma)$ 
such that $\Sim\cl H$ contains $\Sim\calK$. 
We identify $I$ with the set of labels of $\Sim \calK$ 
and $I^c$ with the set
of labels of $\Sim \h\setminus \Sim \calK$ for 
any $\h\supset\calK$. 
Secondly, we apply the simple-projective duality 
\[\iota_{\Gamma}:\EGp(\pvd\Gamma) 
\stackrel{1:1}{\longrightarrow}\SEGp(\per\Gamma)
\] 
of Theorem \ref{thm:KN2}, to get a silting 
object $\bY_\h:=\iota_\Gamma(\calH)$. 
\par
\emph{Step 3.} The
silting object $\bY_\h:=\iota_\Gamma(\calH)\in 
|\SEGp(\pvd(\Gamma))|$ corresponding
to~$\calH$ decomposes as
\begin{equation}\label{eq:Y}
	\bY_\h=:\bigoplus_{j\in I^c}Y_j \oplus\bigoplus_{i\in I}Y_i=:
	\bY_{\cl R}\oplus \bY_\calK.
\end{equation}
Similarly, there is a decomposition
\[\calH=\langle\cl R,\calK\rangle\]
with $\cl R=\<S\mid [S]\in\Sim\h\setminus\Sim\calK\> =\<S_j\mid j\in I^c\>$, 
and $\cl K= \bra S_i\mid i\in I\ket$. 
The summands $Y_\ell$ of $\bY_\h$, and consequently 
of $\bY_\calK$ and $\bY_{\cl R}$, satisfy 
$\Hom_{\per\Gamma}(Y_\ell,S_t)=\delta_{\ell t}\mathbf{k}$ 
(equation \eqref{eq:Hom01}), 
for any $[S_t]\in\Sim\cl H$. 
Therefore 
 \[
	\Hom_{\per\Gamma}(\bY_{\cl R},\calK)\=0 \=\Hom_{\per\Gamma}(\bY_\calK,\cl R),
\]
and $\bY_\calK$ and $\bY_{\cl R}$ are partial silting 
objects in $\per\Gamma$. 
Clearly 
$\thick(\bY_{\cl R})\subseteq
{^{\perp_{\per\Gamma}}\calV}$, because $\cl V =\thick \cl K$, and the 
equality can be proven arguing as in Lemma~\ref{lem:perpV}: 
for any 
$X\in{^{\perp_{\per \Gamma}}\cl V}\subset \per \Gamma$ 
and for any $S_i\in\cl K$
\[0 = \Hom_{\per\Gamma}(X,S_i)=
\oplus\Hom_{\per\Gamma}(Y_\ell[d_\ell],S_i),
\]
where the $Y_\ell$'s are, up to shifts, the summands of 
the minimal perfect 
presentation of $X$, and must be in $\thick \bY_{\cl R}$.
\par
This construction defines a map of sets 
\[\ivc:|\EGb(\pvd(e\Gamma e))|\longrightarrow |\pSEGVb(\per\Gamma)|,\]
diagrammatically represented by
\begin{equation}\label{ivc_diagram}
	\xymatrix@C=.5pt{
		|\EGp(\pvd\Gamma)|\ \ni & \cl H \ar@{|->}[rrrrrr]^-{\iota_\Gamma} &&&&&& \bY_{\cl H}=\bY_{\cl R}\oplus \bY_{\cl K} \ar@{|->}[dd] & \in\ |\SEGp(\per\Gamma)|\\
		|\EGb_{\cl K}(\pvd\Gamma)|\ \ni  \ar@{^{(}->}[u] &\cl H \ar@{=}[u]\\
		|\EGb(\pvd (e\Gamma e))|\ \ni\ar[u]^{1:1} & \ol{\cl H} \ar@{|->}[rrrrrr]^-{\ivc} \ar@{|->}[u] &&&&&& \bY_{\cl R} & \in\ |\pSEGVb(\per \Gamma)|
	}
\end{equation}
\par
\emph{Step 4.} The map $\ivc$ is injective. Suppose 
\begin{equation}\label{hp}\bY_{\cl R}=\ivc(\ol{\cl H}) = \ivc(\ol{\cl H'}).
\end{equation}
Then the 
finite hearts $\cl H$ and $\cl H'$ of $\pvd\Gamma$ are generated 
by simples $\{S_i, S_j\mid i\in I, j\in I^c\}$ 
and $\{S'_i, S_j\mid i\in I, j\in I^c\}$, corresponding by 
hypothesis \eqref{hp} to the same silting object 
$\bY_{\cl H}=\bY_{\cl R}\oplus \bY_{\cl K}=\bY_{\cl H'}$. 
Therefore $S_i\simeq S_i'$ and $\cl H=\cl H'$. 
\par
\emph{Step 5.} We now prove surjectivity, and 
conclude that $\ivc$ is a bijection. \\
We take any $\bY_{\cl R'}=\bigoplus_{j\in I^c} Y'_j$ in 
$|\pSEGVb(\per\Gamma)|$, a completable partial silting object 
such that 
$\thick(\bY_{\cl R'})= {^{\perp_{\per\Gamma}}\cl V}$, and  
complete it to a silting object $\bY'=\bY_{\cl R'}\oplus 
\bY_{\cl K'}$, with
corresponding finite heart $\h'=\<{\cl R'}, \calK'\>
\in |\EGp(\pvd\Gamma)|$. 
The decompositions correspond again to the decomposition of the set 
$Q_0=I\sqcup I^c$. 
Simple-projective duality for $\Gamma$ (Theorem~\ref{thm:KN2}), 
implies that if $M\in\cl H'$, then
$\Hom_{\per\Gamma}(\bY_{\cl R'},M)=0$ if and only if $M\in\calK'$. 
This, together with
$\thick(\bY_{\cl R'})= {^{\perp_{\per\Gamma}}\cl V}=\per(e\Gamma e)$ 
(Lemma \ref{lem:perpV}) implies
that $\h'\cap\calV=\calK'$, which is a Serre subcategory of 
$\h'$.
By Proposition~\ref{prop_antieau}, this means that the heart $\h'$ 
induces a quotient
heart $\ol{\h}$ in $\EGb\left(\pvd(\pvc)\right)$, and we want to 
show that $\bY_{\cl R'}=\ivc\ol{\cl H}$. 
For, we apply the construction of $\ivc$ to $\ol\h$ to 
get a heart in
$\calH=\<\cl R,\calK\>$ in $\EGb_\calK\left(\pvd\Gamma\right)$ (\emph{step 1}), 
with dual silting object
$\bY=\bY_{\cl R}\oplus\bY_{\cl K}$ as in \eqref{eq:Y} (\emph{step 2}), 
and image $\ivc\ol{\cl H}=\bY_{\cl R}$ (\emph{step 3}). 
Let 
\[\Sim\h=\Sim{\cl R}\cup\Sim\calK=
\{[S_j]\}_{j\in I^c}\cup\{[S_i]\}_{i\in I}\]
and 
\[\Sim\h'=\Sim{\cl R'}\cup\Sim\calK'=
\{[S'_j]\}_{j\in I^c}\cup\{[S'_i]\}_{i\in I}.\]
Since $S_j$ and $S_j'$ are mapped to the 
same simples in $\ol{\h}$,
we have $S_j\in \calV*S_j'*\calV$.
And since the summands~$Y_k$ of $\bY_\calK$ are orthogonal
to $\calV$, we conclude that $\Hom_{\per\Gamma}(Y_k,S_j')=
\Hom_{\per\Gamma}(Y_k,S_j)=\delta_{kj}\acf$.
When we regard $\per(e \Gamma e)$ as $\per\bY_{\cl R'}$, 
the object~$Y_k$ admits
a minimal perfect presentation with summands the shifts 
of~$Y_j'$ for $j\in I^c$ (\cite[Lemma 2.14]{plamondon}).
Then the $\Hom$-condition above implies by \eqref{eq:XS} 
that $Y_j=Y_j'$  for any
$j\in I^c$ as required, i.e., $\bY_{\cl R}=\bY_{\cl R'}$. 
This concludes the proof of surjectivity, and hence of bijectivity.
\end{proof}

We now assume that $(Q,W)$ is a quiver with potential 
arising from a triangulation of a (unpunctured) 
decorated marked surface $\surfo$ defined in 
Section \ref{subsec_cate_from_surf}.

\par
\begin{theorem} \label{thm:SPdual}
If $\Gamma=\Gamma(Q,W)$ for some quiver with potential with 
geometric description, then the map
\begin{equation}\label{eq:s-p-dual3}
	\iota_s\colon \EGb(\pvd(\pvc))\to\SEGb(\per(\pvc))
\end{equation}
from Lemma~\ref{le:dual} is an isomorphism between oriented graphs,
and the analogous of \eqref{eq:Hom01} holds.
\end{theorem}

\begin{proof}
The map $\iota_s$ is constructed in the previous proof at the 
level of sets as the composition
\[\iota_s:|\EGb(\pvd(e\Gamma e)) \stackrel{\ivc}{\longrightarrow} 
|\pSEGVb(\per \Gamma)| 
\simeq |\SEGb(\per (e\Gamma e))|.\] 
We recall that the second bijection is an isomorphism of graphs 
by definition.
\par
We identify any silting object in 
$\SEGb(\per (e\Gamma e))$ with the 
corresponding partial silting object in $\pSEGVb(\per\Gamma)$ 
(Lemma \ref{lem:obv}), and we consider a partial 
silting mutation 
\[\bX_1\xrightarrow{Y}(\bX_1)_Y^\sharp=:\bX_2, \text{ in }\pSEGVb(\per\Gamma)\]
with respect to a summand $Y$ of $\bX_1$. 
Arguing as in \emph{Step 5} of the previous proof, we complete 
$\bX_1$ and $\bX_2$ to silting 
objects $\bY_1=\bX_1\oplus \bY_{\cl K}$ and 
$\bY_2=\bX_2\oplus \bY_{\cl K}$ in $|\SEGp(\per\Gamma)|$ and 
we let $\cl H_1:=\bra \cl R_1, \cl K\ket$ and 
$\cl H_2:=\bra \cl R_2, \cl K\ket$ be 
the corresponding hearts in $|\EGp(\pvd\Gamma)|$ under 
the simple projective duality for $\Gamma$ (top left corner 
of \eqref{ivc_diagram}).\\
The mutation formula \eqref{eq:fmu}, together with the condition 
$\Hom_{\per\Gamma}(\bX_1,\cl V)=0$ from Definition \ref{def_pSEG}, 
implies that the silting mutation of $\bY_1=\bX_1\oplus \bY_{\cl K}$ 
at $Y$ does not depend on the summands in $\bY_{\cl K}$, which 
implies that the operations of silting mutation at 
the levels of $\pSEGVb(\per \Gamma)$ and of $\SEGp(\per \Gamma)$ 
commute with the restriction map (the right vertical 
arrow in diagram \eqref{ivc_diagram}):
\[\xymatrix{
	\bY_1 = \bX_1\oplus\bY_{\cl K} \ar[r]^-Y\ar@{|->}[dd]\ar@{}[rdd]|{\scalebox{1.5}{$\square$}}
	        & \mu^\sharp_Y \bY_1=\bX_2\oplus \bY_{\cl K} \ar@{|->}[dd]\\ \\
	\bX_1 \ar[r]^Y & \bX_2
}
\]
We now apply the isomorphism of oriented graphs 
\[\iota_\Gamma^{-1}: \SEGp(\per\Gamma) \to \EGp(\pvd \Gamma)\]
associating the silting mutation
\[\bY_1 \stackrel{Y}{\rightarrow}\bY_2
\] 
to the tilting mutation
\[ \cl H_1 \stackrel{S_Y}{\rightarrow} \cl H_2
\] 
for a simple object $S_Y\in \cl R_1 \subset \cl H_1$.\\
The fact that $\cl H_1 =\bra \cl R_1, \cl K\ket \stackrel{S_Y}{\rightarrow} 
\mu^\sharp_{S_Y}\cl H_1= \cl H_2=\bra \cl R_2, \cl K\ket$ 
does not change the simples in $\cl K$ means that 
$\Ext^1(\cl K, S_Y)=0$ and the hypotheses of 
Proposition \ref{prop_lift_simple_tilt} are satisfied. 
Therefore $\ol{\cl H_2} = \ol{\mu^\sharp_{S_Y}\cl H_1}$ equals $
\mu^\sharp_{\ol{S_Y}}\ol{\cl H_1}$, which shows that for any 
silting mutation there is a uniquely defined corresponding 
simple tilt.
\par
Now let $\ol{\cl H} \in |\EGb(\pvd(e\Gamma e))|$ and 
assume $\cl H\supset \cl K$ as in Jin's correspondence 
\eqref{Jin_map}. We consider a simple tilt 
\begin{equation}\label{st}
	\ol{\cl H} \stackrel{\ol{S}}{\longrightarrow} 
\mu_{\ol{S}}^\sharp\ol{\cl H}
\quad \text{in }\EGb(\pvd(e\Gamma e)).
\end{equation}
Combining Propositions \ref{ExConvRep} 
and \ref{prop_lift_simple_tilt} as in Theorem \ref{quot_graph}, 
the tilt \eqref{st} can be lifted to a simple tilt 
\begin{equation}\label{mm}
\cl H_1 \stackrel{S}{\longrightarrow}{\cl H_2}
\quad \text{in }\EGp(\pvd\Gamma)
\end{equation} 
such that $\ol{\cl H_1}=\ol{\cl H}$, 
$\ol{\cl H_2}=\mu_{\ol{S}}^\sharp\ol{\cl H}$, 
and $\Ext^1(\cl V\cap \cl H_1, S)=0$. 
Under simple-projective duality 
for $\Gamma$ (Theorem \ref{thm:KN2}), the simple 
tilt \eqref{mm} corresponds to a silting mutation 
\begin{equation}\label{oo}
	\bY_1=\bY_{\cl R_1}\oplus\bY_{\cl V\cap \cl H_1} \stackrel{Y_S}{\longrightarrow} 
\bY_2=\bY_{\cl R_2}\oplus\bY_{\cl V\cap \cl H_2},
\end{equation}
with $Y_S$ a summand of $\bY_{\cl R_1}$. The condition 
$\Ext^1(\cl V\cap \cl H_1, S)=0$ implies 
that $(Y_S)^\sharp$ does not depend on 
summands of $\bY_{\cl V\cap \cl H_1}$, therefore the partial silting 
mutation 
$\bY_{\cl R_1}\stackrel{Y_S}{\longrightarrow}\bY_{\cl R_2}$ 
is well-defined.
\par
Note that $\cl H_1$ and $\cl H_2$ need 
not to contain $\cl K$. However, 
by construction (proof of Proposition \ref{ExConvRep}) 
and thanks to the geometric assumption, 
$\cl H_1$ is obtained from $\cl H$ 
by finitely many simple tilts at objects in $\cl K$. 
Therefore, comparing
\[\begin{aligned}
\cl H &\stackrel{\iota_\Gamma}{\longleftrightarrow} 
\bY=\bY_{\cl R}\oplus \bY_{\cl K}, \text{ and}\\
\cl H_1 &\stackrel{\iota_\Gamma}{\longleftrightarrow} 
\bY_1=\bY_{\cl R_1}\oplus \bY_{\cl V\cap \cl H_1},
\end{aligned}\]
we have that $\bY_{\cl R_1}=\bY_{\cl R}=\ivc(\ol{\cl H})$. 
Similarly, it follows from Proposition \ref{cor_lifting_simplefixed}, 
that there exists $\cl G\in|\EGp(\pvd\Gamma)|$ such that 
$\cl G\supset \cl K$, 
$\ol{\cl G} = \mu_{\ol{S}}^\sharp\ol{\cl H}=\ol{\cl H_2}$, 
and with 
$\iota_\Gamma \left(\cl G\right)= \bY_{\cl R_2}\oplus\bY_{\cl K}$. 
Therefore \eqref{oo} restricts to 
\[\ivc(\ol{\cl H}) \stackrel{Y_S}{\longrightarrow} 
\ivc(\mu_{\ol{S}}^\sharp\ol{\cl H}).\]
\end{proof}

\printbibliography

\end{document}